\documentclass[11pt,reqno]{amsart}
\usepackage[utf8]{inputenc}
\usepackage{setspace}
\usepackage{geometry}
\usepackage{enumerate}
\usepackage{enumitem, xcolor, amssymb,latexsym,amsmath,bbm}
\usepackage{mathtools}
\usepackage[mathscr]{euscript}
\usepackage{amsmath}
\usepackage{amssymb}
\usepackage{enumitem}

\usepackage{tikz}
\usepackage{wrapfig}
\usepackage{enumerate}
\usepackage{graphicx}
\usepackage{subfigure}

\usetikzlibrary{arrows.meta}

\usetikzlibrary{decorations.markings}

\usepackage[colorlinks=true,citecolor=blue, linkcolor=blue,urlcolor=blue]{hyperref}

\calclayout

\theoremstyle{plain}
\newtheorem{theorem}{Theorem}[section]

\newtheorem{lemma}[theorem]{Lemma}

\theoremstyle{definition}
\newtheorem{definition}[theorem]{Definition}

\theoremstyle{remark}

\newtheorem{note}[theorem]{Note}

\numberwithin{equation}{section}
\numberwithin{figure}{section}
\allowdisplaybreaks

\title{Length minimization of filling pairs on hyperbolic surfaces}
\author{Ni An, Bhola Nath Saha, Bidyut Sanki}

\date{\today}

\begin{document}
\subjclass[2020]{Primary 57M50; Secondary 57M15, 05C10}

\keywords{}

\begin{abstract}
   A filling pair $(\alpha, \beta)$ of a surface $S_g$ is a pair of simple closed curves in minimal position such that the complement of $\alpha\cup\beta$ in $S_g$ is a disjoint union of topological disks. A filling pair is said to be minimally intersecting if the number of intersections between them, or equivalently, the number of complementary disks, is minimal among all filling pairs of $S_g$. For surfaces of genus $g \geq 3$, minimal filling pairs are well understood, whereas in genus two, such a pair divides the surface into exactly two disks. In this paper, we classify all minimal filling pairs up to the action of the mapping class group in genus two and determine the length of the shortest minimal filling pair.

\end{abstract}

\maketitle

\tikzset{->-/.style={decoration={
  markings,
  mark=at position #1 with {\arrow{>}}},postaction={decorate}}}
  \tikzset{-<-/.style={decoration={
  markings,
  mark=at position #1 with {\arrow{<}}},postaction={decorate}}}
  
\section{Introduction}
Let $S_g$ be a closed, orientable surface of genus $g$.  A \emph{filling system} of $S_g$ is a collection $\Omega=\{ \alpha_1, \dots, \alpha_n \}$ of homotopically distinct simple closed curves whose complement $S_g \setminus \bigcup_i \alpha_i$ is a disjoint union of topological disks. It is assumed that the curves in $\Omega$ are in pairwise minimal position. Filling systems play a fundamental role in the study of surface topology, mapping class groups, and Teichm\"uller spaces. They have various important applications, for instance, in the construction of highly non-trivial mapping class group elements such as pseudo-Anosov maps~\cite{FathiLaudenbachPoenaru2012, Penner}. Thurston~\cite{thurstonspine} has used filling systems in the construction of spines of Teichmüller spaces.

A \emph{filling pair} is a filling system consisting of exactly two curves. A filling pair is said to be \emph{minimally intersecting} or simply \emph{minimal} if the geometric intersection number between the curves, or equivalently, the number of complementary disks, is minimal among all filling pairs. In \cite{AougabMinimally, Sanki2018}, the authors have shown the existence and construction of minimally intersecting filling pairs. They have proved that for $g \geq 3$, the number of complementary disks for a minimal filling pair is one, while for $g = 2$, this number is two. The mapping class group acts on the set of minimally intersecting filling pairs in a natural way.  In \cite{AougabMinimally}, the authors have studied the upper and lower bounds of the number of orbits of this action.

We define the length of a filling system by the following: If $\Omega=\{ \alpha_1, \dots, \alpha_n\}$ is a filling system of a closed hyperbolic surface $X$, then the length of the filling system $\Omega$ is $$\mathrm{length}_X(\Omega)=\sum\limits_{i=1}^nl_X(\alpha_i),$$ where $l_X(\alpha_i)$ denotes the length of the geodesic representative of $\alpha_i$ in the hyperbolic surface $X$. An interesting problem is to study the length of filling systems on hyperbolic surfaces. P. Schmutz \cite{Schaller1999} has shown that for a given filling system, there exists a unique hyperbolic structure where the minimum length of the filling pair occurs. In this direction, the problem of finding the explicit hyperbolic structure on the surface for which the length of a given filling system is minimum is open in general. Another interesting problem is to determine the shortest possible length of minimally intersecting filling pairs, as such pairs have no universal upper bound due to the collar lemma (Section 4.1 \cite{Buser}). Aougab-Huang provided a sharp lower bound on the length of minimally intersecting filling pairs for $g \geq 3$, and one of the main tools to prove it is the isoperimetric inequality for hyperbolic polygons~\cite{Bezdek}.

In \cite{SahaSep}, Saha-Sanki studied minimal filling pairs containing at least one separating curve and determined the shortest possible length of such pairs. Recall that a collection of simple closed curves on a punctured surface is called a \emph{filling pair} if the complement of their union is a disjoint union of disks or once-punctured disks. In a subsequent work~\cite{SahaSanki2024}, the authors investigated the length of minimal filling pairs on once-punctured hyperbolic surfaces. 

In genus two, the main obstacle to extending the Aougab-Huang approach for determining the shortest length of minimal filling pairs lies in the fact that a minimal filling pair has two complementary regions, and the number of edges in these regions is not uniquely determined. Note that for a filling pair $(\alpha, \beta)$, the number of edges of a region corresponds to the total number of $\alpha$- and $\beta$-arcs along its boundary. An Euler characteristic argument shows that the total number of sides of the two regions is $16$, leading three possible configurations: $\{4,12\}$, $\{6,10\}$, and $\{8,8\}$. We refer to filling pairs corresponding to these cases as being of type $\{4,12\}$, $\{6,10\}$, and $\{8,8\}$, respectively. As a first step toward understanding the structure of minimal fillings in genus two, we eliminate the possibility of the configuration of type $\{6,10\}$. In particular, we establish the following proposition.

\begin{theorem}\label{prop1.1}
   A minimally intersecting filling pair $(\alpha, \beta)$ of a closed surface $S_2$ of genus two is of either type $\{4, 12\}$ or type $\{8,8\}$ (equivalently, there does not exist a filling pair of type $\{6,10\}$ on $S_2$).
\end{theorem}

The proof of Theorem~\ref{prop1.1} relies on the theory of fat graphs (a formal definition of fat graphs is provided in the next section). 

Let $\Sigma$ denote the set of all minimally intersecting filling pairs on $S_2$. The mapping class group $\mathrm{Mod}(S_2)$ acts on $\Sigma$ in the following way: For $f\in \mathrm{Mod}(S_2)$ and $(\alpha, \beta)\in \Sigma$, we have $f\cdot (\alpha, \beta)=(f(\alpha), f(\beta))$. In the course of proving Theorem~\ref{prop1.1}, we obtain the following theorem.

\begin{theorem}\label{thm1.2}
    There are exactly two (up to mapping class group action) minimal filling pairs in $S_2$.
\end{theorem}

The proof of Theorem~\ref{thm1.2} is based on fat graph theory and follows the approach of Aougab and Huang (see the proof of Lemma~2.3 in~\cite{AougabMinimally}). This theorem provides a significant advantage in minimizing the length of minimal filling pairs. We address the length problem in two distinct cases: the $\{4,12\}$ case and the $\{8,8\}$ case, which are presented in the following two theorems.

\begin{theorem}\label{thm1.3}
    Let $(\alpha, \beta)$ be a $\{8,8\}$ filling pair in a hyperbolic surface $X\in \mathcal{T}_2$. Then
    $$\mathrm{length}_X(\alpha,\beta) \geq 8\cdot\cosh^{-1}\left(\sqrt{2}+1\right)\approx 12.228567.$$ 
\end{theorem}
The proof of Theorem~\ref{thm1.3} is based on studying the angles at the intersection points of the geodesic representatives of the curves and applying the isoperimetric inequality for hyperbolic polygons.

In Theorem \ref{thm1.3}, we have seen that the minimal length of $\{8,8\}$ filling pair is given by a regular configuration. Surprisingly, this configuration does not provide the global minima of the length minimal filling pairs in genus two. We will see in the following theorem that the minima occurs in the $\{4,12\}$ case.
 
 \begin{theorem}\label{thm1.4}
    Let $(\alpha, \beta)$ be a $\{4,12\}$ filling pair in a hyperbolic surface $X\in \mathcal{T}_2$. Then
    $$\mathrm{length}_X(\alpha,\beta) \geq L_0,$$
    where $L_0=6 \cdot \cosh^{-1}(7/2)\approx 11.5490838.$
\end{theorem}

 The proof of Theorem~\ref{thm1.4} builds on these ideas and also uses numerical computations. Finally, we show that the lower bound obtained in Theorem~\ref{thm1.4} gives a global lower bound for any filling pair on a genus-two surface.

\section{Preliminaries}
In this section, we recall some definitions that will be frequently used throughout the paper.
We first recall the definition of graphs. Although it is not the standard one, it is equivalent and well-suited for defining fat graphs. We conclude this section with Lemma \ref{lem: fat graph boundary} that deals with boundary components of a fat graph.
\begin{definition}[Graph]
A graph $G$ is a triple $G=(E,\sim,\sigma_1)$, where
\begin{enumerate}
    \item $E=\{e_1,e_1^{-1},\dots, e_n,e_n^{-1}\}$ is a finite, non-empty set, called the set of directed edges.
    \item $\sim$ is an equivalence relation on $E$.
    \item $\sigma_1:E\to E$, a fixed point free involution, maps a directed edge to its reverse directed edge, i.e., $\sigma_1(e_i)=e_i^{-1},$ for all $e_i\in E$.
\end{enumerate}
\end{definition}

The equivalence relation $\sim$ is defined by $e_1\sim e_2$ if $e_1$ and $e_2$ have same initial vertex. In an ordinary language, $V=E/\sim$ is the set of all vertices and $E/\sigma_1$ is the set of un-directed edges of the graph. The number of edges incident at a vertex is called the \textit{degree} of the vertex (for more details, we refer to~\cite{Sanki2018}, Section 2). Now, we define fat graphs in the following.
\begin{definition}
A \textit{fat graph} is a quadruple $\Gamma=(E,\sim,\sigma_1,\sigma_0)$, where
\begin{enumerate}
    \item $G=(E,\sim,\sigma_1)$ is a graph.
    \item $\sigma_0$ is a permutation on E so that each cycle corresponds to a cyclic order on the set of oriented edges going out from a vertex.
\end{enumerate}
\end{definition}

\begin{definition}[Standard cycle]
    A closed path in a fat graph of even valency is called a \textit{standard cycle} if every two consecutive edges in the cycle are opposite to each other in the cyclic order on the set of edges incident at their common vertex.
\end{definition}

\textbf{Surface associated to a fat graph.}
Given a fat graph $\Gamma=(E, \sim, \sigma_1, \sigma_0)$, we construct an oriented topological surface $\Sigma(\Gamma)$ with boundary components as follows: Consider a closed disk corresponding to each vertex and a rectangle corresponding to each un-directed edge. Then identify the sides of the rectangles with the boundary of the disks according to the order of the edges incident to the vertex. See Figure~\ref{loc_pic_of_fat_graph} for a local picture. We define, the boundary of a fat graph $\Gamma$ as the boundary of the surface $\Sigma(\Gamma)$. We get a closed surface from $\Sigma(\Gamma)$ by attaching topological disks along the boundaries.

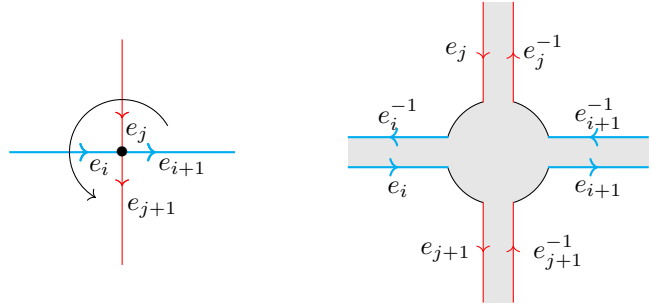
\begin{figure}[htbp]
\begin{center}
\begin{tikzpicture}[xscale=1,yscale=1]
\draw[thick, cyan, ->-=.35, ->-=.65] (-1.5,0)--(1.5,0);
\draw [red, ->-=.35, ->-=.65] (0,1.5)--(0,-1.5);
\draw[->-=1] (0.6062,.35) arc (30:240:.7);
\draw (-.3,-.2) node {\small$e_i$} (.8,-.2) node {\small$e_{i+1}$} (.2,.25) node {\small$e_j$} (.4,-.75) node {\small$e_{j+1}$};
\draw[] (0,0) node {\small$ \bullet$};

\fill[gray!20!white] (5,0) circle (.7cm);
\draw [] (5,0) circle (.7cm);
\fill[gray!20!white] (3,.2)--(4.34,.2)--(4.34,-.2)--(3,-.2)--cycle;
\draw [thick, cyan, -<-=.5] (3,.2)--(4.34,.2);
\draw [thick, cyan, ->-=.5] (3,-.2)--(4.34,-.2);
\fill [gray!20!white] (5.66,.2)--(7,.2)--(7,-.2)--(5.66,-.2)--cycle;
\draw  [thick, cyan, -<-=.5] (5.66,.2)--(7,.2);
\draw [thick, cyan, ->-=.5] (5.66,-.2)--(7,-.2);
\fill [gray!20!white] (4.8,.66)--(4.8,2)--(5.2,2)--(5.2,.66)--cycle;
\draw [red, -<-=.5] (4.8,.66)--(4.8,2);
\draw [red, ->-=.5] (5.2,.66)--(5.2,2);
\fill [gray!20!white] (4.8,-.66)--(4.8,-2)--(5.2,-2)--(5.2,-.66)--cycle;
\draw [red, ->-=.5] (4.8,-.66)--(4.8,-2);
\draw [red, -<-=.5] (5.2,-.66)--(5.2,-2);

    \draw (3.67,.5) node {\small$e_i^{-1}$} (3.67,-.5) node {\small$e_{i}$} (6.33,.5) node {\small$e_{i+1}^{-1}$} (6.33,-.5) node {\small$e_{i+1}$};
    \draw (5.6,1.33) node {\small$e_j^{-1}$} (4.45,1.33) node {\small$e_{j}$} (5.8,-1.33) node {\small$e_{j+1}^{-1}$} (4.3,-1.33) node {\small$e_{j+1}$};
\end{tikzpicture}
\end{center}
 \caption{Local picture of the surface obtained from a fat graph.} 
\label{loc_pic_of_fat_graph}
\end{figure}

\begin{lemma}\label{lem: fat graph boundary}
The number of boundary components of a fat graph $\Gamma=(E, \sim, \sigma_1, \sigma_0)$ is the same as the number of disjoint cycles in $\sigma_0\sigma_1$. Moreover, under suitable labeling, the cycles of $\sigma_0\sigma_1$ and the boundary words of $\Gamma$ coincide.
\end{lemma}

\begin{proof}
    Let us view the fat graph locally (see Figure \ref{bdry_of_fat_graph_calculation}). For convenience, we assume that the valency of the graph is 4. The proof for the other valencies follows similarly. We adopt the convention for labeling the boundary words of $\Sigma(\Gamma)$ illustrated in Figure \ref{bdry_of_fat_graph_calculation}. More precisely, upon traversing along an edge $\alpha_i$, the boundary on the right is labeled as $\alpha_i$, whereas the boundary on the left is labeled by $\alpha_i^{-1}$.
    
    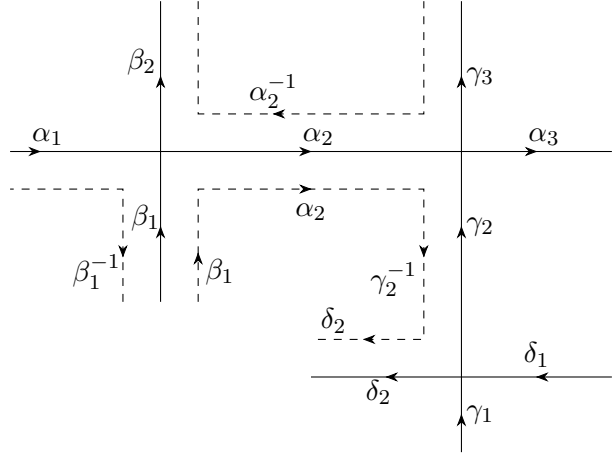
\begin{figure}[htbp]
    \begin{center}
    \begin{tikzpicture}[xscale=1,yscale=1]
     
    \draw (0,-2) -- (0,0) -- (4,0) -- (4,-3) -- (4,-4) ;
    \draw (-2,0) -- (0,0)--(0,2);
    \draw (4,2)--(4,0)--(6,0);
    \draw (2,-3)--(6,-3);
    \draw [-{Stealth[color=black]}] (-1.6,0)--(-1.59,0);
    \draw [-{Stealth[color=black]}] (5,0)--(5.01,0);
    \draw [-{Stealth[color=black]}] (2,0)--(2.01,0);
    \draw [-{Stealth[color=black]}] (0,-1)--(-0,-.99);
    \draw [-{Stealth[color=black]}] (0,1)--(-0,1.01);
    \draw [-{Stealth[color=black]}] (4,1)--(4,1.01);
    \draw [-{Stealth[color=black]}] (4,-1)--(4,-.99);
    \draw [-{Stealth[color=black]}] (4,-3.5)--(4,-3.49);
    \draw [-{Stealth[color=black]}] (3,-3)--(2.99,-3);
    \draw [-{Stealth[color=black]}] (5,-3)--(4.99,-3);
    \draw (-1.5,.2) node {$\alpha_1$} (2.1,.2) node {$\alpha_2$} (5.1,.2) node {$\alpha_3$} (-.2,-.88) node{$\beta_1$} (-.25,1.2) node{$\beta_2$} (4.25,-3.5) node{$\gamma_1$} (4.25,-1) node{$\gamma_2$} (4.25,1) node{$\gamma_3$} (5,-2.7) node{$\delta_1$} (2.9,-3.2) node{$\delta_2$};
    \draw[dashed] (.5,-2) -- (.5,-.5) -- (3.5,-.5) -- (3.5,-2.5) -- (2,-2.5);
    \draw[dashed] (-.5,-2) -- (-.5,-.5) -- (-2,-.5);
    \draw[dashed] (.5,2) -- (.5,.5) -- (3.5,.5) -- (3.5,2);
    \draw [-{Stealth[color=black]}] (-.5,-1.4)--(-.5,-1.41);
    \draw [-{Stealth[color=black]}] (.5,-1.35)--(.5,-1.34);
    \draw [-{Stealth[color=black]}] (2,-.5)--(2.01,-.5);
    \draw [-{Stealth[color=black]}] (3.5,-1.4)--(3.5,-1.41);
    \draw [-{Stealth[color=black]}] (2.7,-2.5)--(2.690,-2.5);
    \draw [-{Stealth[color=black]}] (1.5,.5)--(1.49,.5);
    \draw (.79,-1.6) node {$\beta_1$} (-.85,-1.6) node {$\beta_1^{-1}$} (2,-.8) node {$\alpha_2$} (3.1,-1.7) node {$\gamma_2^{-1}$} (2.3,-2.25) node {$\delta_2$} (1.5,.8) node {$\alpha_2^{-1}$};
    \end{tikzpicture}
    \end{center}
    \caption{The boundary of the surface is oriented so that the surface part lies on the right side.}
    \label{bdry_of_fat_graph_calculation}
    \end{figure}
    
    According to Figure \ref{bdry_of_fat_graph_calculation}, the permutation $\sigma_0$ is of the form 
    $$\sigma_0= (\alpha_2, \beta_2, \alpha_1^{-1}, \beta_1^{-1}) (\alpha_3, \gamma_3, \alpha_2^{-1}, \gamma_2^{-1}) (\delta_1^{-1}, \gamma_2, \delta_2, \gamma_1^{-1}) \cdots.$$
    
    Now, we evaluate the cycle containing $\beta_1$ in the permutation $\sigma_0\sigma_1$ and compare this cycle with the boundary word of $\Sigma(\Gamma)$ containing $\beta_1$. It can be seen that the cycle in $\sigma_0\sigma_1$ containing $\beta_1$ is $(\beta_1, \alpha_2, \gamma_2^{-1}, \cdots)$ (see Figure \ref{action_of_sigma0_inv_sigma_1}). Thus, the boundary word and cycle coincide, and this completes the proof.
    
    \begin{figure}[htbp]
    \begin{center}
    \begin{tikzpicture}[xscale=1,yscale=1]
    
    \draw (0.1,0) node{$\beta_1$};               
    \draw [-{Stealth[color=black]}] (.5,0)--(2,0);
    \draw (2.5,0) node {$\beta_1^{-1}$};
    \draw [-{Stealth[color=black]}] (3,0)--(4.5,0);
    \draw (4.9,0) node {$\alpha_2$};
    \draw [-{Stealth[color=black]}] (5.3,0)--(6.8,0);
    \draw (7.3,0) node {$\alpha_2^{-1}$};
    \draw [-{Stealth[color=black]}] (7.8,0)--(9.3,0);
    \draw (10,0) node {$\gamma_2^{-1} \dots$};  
    \draw (1.2,.2) node {$\sigma_1$} (3.75,.28) node {$\sigma_0$} (6,.2) node {$\sigma_1$} (8.5,.28) node {$\sigma_0$};
    \end{tikzpicture}
    \end{center}
    \caption{The action of $\sigma_0\sigma_1$.}
    \label{action_of_sigma0_inv_sigma_1}
    \end{figure}
\end{proof}

\section{Classification of minimal filling pairs up to mapping class group action}
In this section, first we classify all minimal filling pairs on $S_2$ in terms of the number of edges in the complementary polygons. The mapping class group $\mathrm{Mod}(S_2)$ acts on the set of all minimal filling pairs as follows: $f \cdot (\alpha,\beta)=(f(\alpha), f(\beta))$, where $f \in \mathrm{Mod}(S_2)$ and $(\alpha,\beta)$ is a minimal filling pair on $S_2$. We conclude by determining exactly the number of orbits of this action.

Let $(\alpha, \beta)$ be a minimal filling pair on $S_2$ and
$S_2\setminus (\alpha\cup \beta)=D_1 \sqcup D_2.$
By Euler characteristic formula, it follows that $i(\alpha,\beta)=4$. It is easy to see that $|D_1|+|D_2|=16$, where $|D_i|$ denotes the number of edges in $D_i, i=1,2$. As $|D_1|,|D_2| \in \{ 4,6,8,10,12\}$, $\{|D_1|,|D_2|\}=\{4,12\} \text{ or } \{6,10\} \text{ or } \{8,8\}$. We call a minimal filling pair as \emph{of type} $\{ 4,12 \}$ if the complementary disks contain 4 and 12 edges. A filling pair of type $\{6,10\} \text{ or } \{8,8\}$ is similarly defined. First, we prove that the case $\{6,10\}$ cannot occur.

\begin{theorem}
   A minimally intersecting filling pair $(\alpha, \beta)$ of a closed surface $S_2$ of genus two is of either type $\{4, 12\}$ or type $\{8,8\}$ (equivalently, there does not exist a filling pair of type $\{6,10\}$ on $S_2$).
\end{theorem}

\begin{proof}
    We may regard the union $\alpha \cup \beta$ as a fat graph, where the vertices correspond to the points of intersection $\alpha \cap \beta$, the edges are the arcs of $\alpha$ and $\beta$ between consecutive intersection points, and the cyclic order at each vertex is determined by the orientation of the surface. We fix a direction and choose an initial arc of $\alpha$, labeling it $\alpha_1$. The remaining arcs of $\alpha$ are then labeled consecutively along this direction as $\alpha_2, \alpha_3,$ and $\alpha_4$. Similarly, we label the arcs of $\beta$ by $\beta_1, \dots, \beta_4$.

    To prove the proposition, we explicitly determine all possible fat graphs having two standard cycles and four vertices. These configurations will serve as a basis for characterizing all minimal filling pairs in $S_2$. We perform this analysis in eight different cases.

    \begin{enumerate}
        \item As a first case, we consider the fat graph in Figure \ref{fig: fatg case1}. The permutations $\sigma_0$ and the boundary word $\sigma_0^{-1}\sigma_1$ are given by 
    \begin{align*}
        \sigma_0 & = (\alpha_1, \beta_4^{-1}, \alpha_4^{-1}, \beta_1) (\alpha_2, \beta_3, \alpha_1^{-1}, \beta_2^{-1}) (\alpha_3, \beta_2, \alpha_2^{-1}, \beta_1^{-1}) (\alpha_4, \beta_3^{-1}, \alpha_3^{-1}, \beta_4),\\
        \sigma_0\sigma_1 & = (\alpha_1, \beta_2^{-1}, \alpha_2^{-1}, \beta_3, \alpha_3^{-1}, \beta_2, \alpha_2, \beta_1^{-1})(\alpha_1^{-1}, \beta_4^{-1}, \alpha_4, \beta_1, \alpha_3, \beta_4, \alpha_4^{-1}, \beta_3^{-1}).
    \end{align*}
    
    \begin{figure}[htbp]
        \centering
        \begin{tikzpicture}[xscale=.7,yscale=.7]
            \draw [-<-=.98, -<-=.93, -<-=.865, -<-=.785, -<-=.74] (0,1)--(6,1)--(6,10)--(0,10)--cycle;
            
            \draw (.32,9.1) node {\tiny$\alpha_1$}  (.32,3.5) node {\tiny$\alpha_2$} (.32,5.4) node {\tiny$\alpha_3$} (.33,7.5) node {\tiny$\alpha_4$};
            
    
            \draw [->-=.2, ->-=.75] (-1.5,8)--(1.5,8);
            \draw (1,8.3) node {\tiny $\beta_1$};

            \draw [-<-=.2, -<-=.75] (-1.5,2)--(1.5,2);
            \draw [-<-=.2, -<-=.75] (-1.5,4)--(1.5,4);
            \draw [->-=.2, ->-=.75] (-1.5,6)--(1.5,6);

            \draw (-1,2.3) node {\tiny $\beta_3$} (-1,4.3) node {\tiny $\beta_2$} (1,6.3) node {\tiny $\beta_4$};
    
            \draw (-1.5,8)--(-1.5,7)--(-.2,7);
            \draw (.2,7)--(1.5,7)--(1.5,6);
            \draw (-1.5,6)--(-3,6)--(-3,2)--(-1.5,2);
            \draw (1.5,2)--(1.5,3)--(.2,3);
            \draw (-.2,3)--(-1.5,3)--(-1.5,4);
            \draw (1.5,4)--(3,4)--(3,8)--(1.5,8);

              \draw [bend left, -<-=.5] ({9.5+2*cos(22.7)},{7.5+2*sin(22.7)}) to ({9.5+2*cos(67.5)},{7.5+2*sin(67.5)});
              \draw [bend left, -<-=.5] ({9.5+2*cos(67.5)},{7.5+2*sin(67.5)}) to ({9.5+2*cos(112.5)},{7.5+2*sin(112.5)});
              \draw [bend left, -<-=.5] ({9.5+2*cos(112.5)},{7.5+2*sin(112.5)}) to ({9.5+2*cos(157.5)},{7.5+2*sin(157.5)});
              \draw [bend left, -<-=.5] ({9.5+2*cos(157.5)},{7.5+2*sin(157.5)}) to ({9.5+2*cos(202.5)},{7.5+2*sin(202.5)});
              \draw [bend left, -<-=.5] ({9.5+2*cos(202.5)},{7.5+2*sin(202.5)}) to ({9.5+2*cos(247.5)}, {7.5+2*sin(247.5)});
              \draw [bend left, -<-=.5] ({9.5+2*cos(247.5)}, {7.5+2*sin(247.5)}) to ({9.5+2*cos(292.5)},{7.5+2*sin(292.5)});
              \draw [bend left, -<-=.5] ({9.5+2*cos(292.5)},{7.5+2*sin(292.5)}) to ({9.5+2*cos(337.5)},{7.5+2*sin(337.5)});
              \draw [bend left, -<-=.5] ({9.5+2*cos(337.5)},{7.5+2*sin(337.5)}) to ({9.5+2*cos(22.7)},{7.5+2*sin(22.7)});
    
              \draw  ({9.5+1.9*cos(90)},{7.5+1.9*sin(90)}) node {\tiny$\alpha_1$} ({9.5+2*cos(135)},{7.5+2*sin(135)}) node {\tiny$\beta_1^{-1}$} ({9.5+2*cos(180)},{7.5+2*sin(180)}) node {\tiny$\alpha_2$} ({9.5+2*cos(225)},{7.5+2*sin(225)}) node {\tiny$\beta_2$} ({9.5+2*cos(270)},{7.5+2*sin(270)}) node {\tiny$\alpha_3^{-1}$} ({9.5+2*cos(315)},{7.5+2*sin(315)}) node {\tiny$\beta_3$} ({9.5+2*cos(0)},{7.5+2*sin(0)}) node {\tiny$\alpha_2^{-1}$} ({9.5+2*cos(45)},{7.5+2*sin(45)}) node {\tiny$\beta_2^{-1}$};

              \draw [bend left, -<-=.5] ({9.5+2*cos(22.7)},{3+2*sin(22.7)}) to ({9.5+2*cos(67.5)},{3+2*sin(67.5)});
              \draw [bend left, -<-=.5] ({9.5+2*cos(67.5)},{3+2*sin(67.5)}) to ({9.5+2*cos(112.5)},{3+2*sin(112.5)});
              \draw [bend left, -<-=.5] ({9.5+2*cos(112.5)},{3+2*sin(112.5)}) to ({9.5+2*cos(157.5)},{3+2*sin(157.5)});
              \draw [bend left, -<-=.5] ({9.5+2*cos(157.5)},{3+2*sin(157.5)}) to ({9.5+2*cos(202.5)},{3+2*sin(202.5)});
              \draw [bend left, -<-=.5] ({9.5+2*cos(202.5)},{3+2*sin(202.5)}) to ({9.5+2*cos(247.5)}, {3+2*sin(247.5)});
              \draw [bend left, -<-=.5] ({9.5+2*cos(247.5)}, {3+2*sin(247.5)}) to ({9.5+2*cos(292.5)},{3+2*sin(292.5)});
              \draw [bend left, -<-=.5] ({9.5+2*cos(292.5)},{3+2*sin(292.5)}) to ({9.5+2*cos(337.5)},{3+2*sin(337.5)});
              \draw [bend left, -<-=.5] ({9.5+2*cos(337.5)},{3+2*sin(337.5)}) to ({9.5+2*cos(22.7)},{3+2*sin(22.7)});
    
              \draw  ({9.5+1.9*cos(90)},{3+1.9*sin(90)}) node {\tiny$\alpha_1^{-1}$} ({9.5+2*cos(135)},{3+2*sin(135)}) node {\tiny$\beta_3^{-1}$} ({9.5+2*cos(180)},{3+2*sin(180)}) node {\tiny$\alpha_4^{-1}$} ({9.5+2*cos(225)},{3+2*sin(225)}) node {\tiny$\beta_4$} ({9.5+2*cos(270)},{3+2*sin(270)}) node {\tiny$\alpha_3$} ({9.5+2*cos(315)},{3+2*sin(315)}) node {\tiny$\beta_1$} ({9.5+2*cos(0)},{3+2*sin(0)}) node {\tiny$\alpha_4$} ({9.5+2*cos(45)},{3+2*sin(45)}) node {\tiny$\beta_4^{-1}$};
        \end{tikzpicture}
        \caption{Case 1}
        \label{fig: fatg case1}
    \end{figure}
        \item \label{case1} The permutations $\sigma_0$ and $\sigma_0^{-1}\sigma_1$ are given by (see Figure \ref{fig: fatg case2})
    \begin{align*}
        \sigma_0 & = (\alpha_1, \beta_4^{-1}, \alpha_4^{-1}, \beta_1) (\alpha_2, \beta_2^{-1}, \alpha_1^{-1}, \beta_3) (\alpha_3, \beta_1^{-1}, \alpha_2^{-1}, \beta_2) (\alpha_4, \beta_3^{-1}, \alpha_3^{-1}, \beta_4),\\
        \sigma_0\sigma_1 & = (\alpha_1, \beta_3, \alpha_3^{-1}, \beta_1^{-1}) (\alpha_1^{-1}, \beta_4^{-1}, \alpha_4, \beta_1, \alpha_2^{-1}, \beta_2^{-1}, \alpha_3, \beta_4, \alpha_4^{-1}, \beta_3^{-1}, \alpha_2, \beta_2).
    \end{align*}

    \begin{figure}[htbp]
        \centering
        \begin{tikzpicture}[xscale=.7,yscale=.7]
            \draw [-<-=.68, -<-=.93, -<-=.865, -<-=.785, -<-=.74] (0,1)--(4,1)--(4,10)--(0,10)--cycle;
            
            \draw (.32,9.4) node {\tiny$\alpha_1$}  (.32,3.5) node {\tiny$\alpha_2$} (.32,5.4) node {\tiny$\alpha_3$} (.33,7.5) node {\tiny$\alpha_4$};
    
            \draw [->-=.2, ->-=.75] (-1.5,8)--(1.5,8);
            \draw [->-=.2, ->-=.75] (-1.5,2)--(1.5,2);
            \draw [->-=.2, ->-=.75] (-1.5,4)--(1.5,4);
            \draw [->-=.2, ->-=.75] (-1.5,6)--(1.5,6);
    
            \draw (1,8.3) node { \tiny$\beta_1$} (1,2.3) node { \tiny$\beta_3$} (1,4.3) node { \tiny$\beta_2$} (1,6.3) node {\tiny$\beta_4$};
             \draw (-1.5,8)--(-1.5,7)--(-.2,7);
             \draw (.2,7)--(1.5,7)--(1.5,6)--(-4.5,6)--(-4.5,3)--(-.2,3);
             \draw (.2,3)--(1.5,3)--(1.5,2)--(-3,2)--(-3,2.8);
             \draw (-3,3.2)--(-3,5)--(-.2,5);
             \draw (.2,5)--(1.5,5)--(1.5,4)--(-2,4)--(-2,4.8);
             \draw (-2,5.2)--(-2,5.8);
             \draw (-2,6.2)--(-2,9);
             \draw (-2,9)--(-.2,9);
             \draw (.2,9)--(1.5,9)--(1.5,8);
    
             \foreach \x in {1,2,...,12}
             {
             \draw [bend left, -<-=.5] ({7.5+2.5*cos(15+\x*30)}, {7+2.5*sin(15+\x*30)}) to ({7.5+2.5*cos(15+(\x+1)*30)}, {7+2.5*sin(15+(\x+1)*30)});
             }
    
             \draw ({7.5+2.6*cos(0)}, {7+2.6*sin(0)}) node {\tiny$\alpha_3$} ({7.5+2.7*cos(30)}, {7+2.7*sin(30)}) node {\tiny$\beta_2^{-1}$} ({7.5+2.6*cos(60)}, {7+2.6*sin(60)}) node {\tiny$\alpha_2^{-1}$} ({7.5+2.6*cos(90)}, {7+2.6*sin(90)}) node {\tiny$\beta_1$} ({7.5+2.6*cos(120)}, {7+2.6*sin(120)}) node {\tiny$\alpha_4$} ({7.5+2.7*cos(150)}, {7+2.7*sin(150)}) node {\tiny$\beta_4^{-1}$} ({7.5+2.6*cos(180)}, {7+2.6*sin(180)}) node {\tiny$\alpha_1^{-1}$} ({7.5+2.6*cos(210)}, {7+2.6*sin(210)}) node {\tiny$\beta_2$} ({7.5+2.6*cos(240)}, {7+2.6*sin(240)}) node {\tiny$\alpha_2$} ({7.5+2.7*cos(270)}, {7+2.7*sin(270)}) node {\tiny$\beta_3^{-1}$} ({7.5+2.6*cos(300)}, {7+2.6*sin(300)}) node {\tiny$\alpha_4^{-1}$} ({7.5+2.6*cos(330)}, {7+2.6*sin(330)}) node {\tiny$\beta_4$};
    
             \draw [-<-=.125, -<-=.375, -<-=.625, -<-=.875] (6.5,1.5)--(8.5,1.5)--(8.5,3.5)--(6.5,3.5)--cycle;
             \draw (7.3,3.75) node {\tiny$\alpha_1$} (7.7,1.2) node {\tiny$\alpha_3^{-1}$} (8.9,2.5) node {\tiny$\beta_3$} (6,2.5) node {\tiny$\beta_1^{-1}$}; 
        \end{tikzpicture}
        \caption{Case 2}
        \label{fig: fatg case2}
    \end{figure}

    \item The permutations $\sigma_0$ and $\sigma_0^{-1}\sigma_1$ are given by (see Figure \ref{fig: fatg case 3})
    \begin{align*}
        \sigma_0 & = (\alpha_1, \beta_4^{-1}, \alpha_4^{-1}, \beta_1) (\alpha_2, \beta_1^{-1}, \alpha_1^{-1}, \beta_2) (\alpha_3, \beta_4, \alpha_2^{-1}, \beta_3^{-1}) (\alpha_4, \beta_3, \alpha_3^{-1}, \beta_2^{-1}),\\
        \sigma_0\sigma_1 & = (\alpha_1, \beta_2, \alpha_4, \beta_1, \alpha_1^{-1}, \beta_4^{-1}, \alpha_2^{-1}, \beta_1^{-1})( \alpha_2, \beta_3^{-1}, \alpha_3^{-1}, \beta_4, \alpha_4^{-1}, \beta_3, \alpha_3, \beta_2^{-1}).
    \end{align*}

    \begin{figure}[htbp]
        \centering
        \begin{tikzpicture}[xscale=.7,yscale=.7]
            \draw [-<-=.98, -<-=.93, -<-=.865, -<-=.785, -<-=.74] (0,1)--(6,1)--(6,10)--(0,10)--cycle;
            
            \draw (.32,9.1) node {\tiny$\alpha_1$}  (.32,3.5) node {\tiny$\alpha_2$} (.32,5.4) node {\tiny$\alpha_3$} (.33,7.5) node {\tiny$\alpha_4$};
    
            \draw [->-=.2, ->-=.75] (-1.5,8)--(1.5,8);
            \draw (1,8.3) node {\tiny $\beta_1$};

            \draw [->-=.2, ->-=.75] (-1.5,2)--(1.5,2);
            \draw [-<-=.2, -<-=.75] (-1.5,4)--(1.5,4);
            \draw [-<-=.2, -<-=.75] (-1.5,6)--(1.5,6);

            \draw (1,2.3) node {\tiny $\beta_2$} (-1,4.3) node {\tiny $\beta_4$} (-1,6.3) node {\tiny $\beta_3$};

            \draw (1.5,8)--(3,8)--(3,3)--(.2,3);
            \draw (-.2,3)--(-1.5,3)--(-1.5,2);
            \draw (1.5,2)--(4.5,2)--(4.5,6)--(3.2,6);
            \draw (2.8,6)--(1.5,6);
            \draw (-1.5,6)--(-1.5,5)--(-.2,5);
            \draw (.2,5)--(1.5,5)--(1.5,4);
            \draw (-1.5,4)--(-3,4)--(-3,8)--(-1.5,8);

            \foreach \x in {1,2,...,8}
            {
            \draw [bend left,-<-=.5] ({9.5+2*cos(\x*45+22.5)}, {7.5+2*sin(\x*45+22.5)}) to ({9.5+2*cos((\x+1)*45+22.5)}, {7.5+2*sin((\x+1)*45+22.5)});
            }

            \draw ({9.5+1.9*cos(90)},{7.5+1.9*sin(90)}) node {\tiny$\alpha_1$} ({9.5+2*cos(135)},{7.5+2*sin(135)}) node {\tiny$\beta_1^{-1}$} ({9.5+2*cos(180)},{7.5+2*sin(180)}) node {\tiny$\alpha_2^{-1}$} ({9.5+2*cos(225)},{7.5+2*sin(225)}) node {\tiny$\beta_4^{-1}$} ({9.5+2*cos(270)},{7.5+2*sin(270)}) node {\tiny$\alpha_1^{-1}$} ({9.5+2*cos(315)},{7.5+2*sin(315)}) node {\tiny$\beta_1$} ({9.5+2*cos(0)},{7.5+2*sin(0)}) node {\tiny$\alpha_4$} ({9.5+2*cos(45)},{7.5+2*sin(45)}) node {\tiny$\beta_2$};

            \foreach \x in {1,2,...,8}
            {
            \draw [bend left,-<-=.5] ({9.5+2*cos(\x*45+22.5)}, {3+2*sin(\x*45+22.5)}) to ({9.5+2*cos((\x+1)*45+22.5)}, {3+2*sin((\x+1)*45+22.5)});
            }

            \draw  ({9.5+1.9*cos(90)},{3+1.9*sin(90)}) node {\tiny$\alpha_2$} ({9.5+2*cos(135)},{3+2*sin(135)}) node {\tiny$\beta_2^{-1}$} ({9.5+2*cos(180)},{3+2*sin(180)}) node {\tiny$\alpha_3$} ({9.5+2*cos(225)},{3+2*sin(225)}) node {\tiny$\beta_3$} ({9.5+2*cos(270)},{3+2*sin(270)}) node {\tiny$\alpha_4^{-1}$} ({9.5+2*cos(315)},{3+2*sin(315)}) node {\tiny$\beta_4$} ({9.5+2*cos(0)},{3+2*sin(0)}) node {\tiny$\alpha_3^{-1}$} ({9.5+2*cos(45)},{3+2*sin(45)}) node {\tiny$\beta_3^{-1}$};
        \end{tikzpicture}
        \caption{Case 3}
        \label{fig: fatg case 3}
    \end{figure}

    \item The permutations $\sigma_0$ and $\sigma_0^{-1}\sigma_1$ are given by 
    \begin{align*}
        \sigma_0 & = (\alpha_1, \beta_4^{-1}, \alpha_4^{-1}, \beta_1) (\alpha_2, \beta_3, \alpha_1^{-1}, \beta_2^{-1}) (\alpha_3, \beta_4, \alpha_2^{-1}, \beta_3^{-1}) (\alpha_4, \beta_1^{-1}, \alpha_3^{-1}, \beta_2),\\
        \sigma_0\sigma_1 & = (\alpha_1, \beta_2^{-1}, \alpha_4, \beta_1, \alpha_3^{-1}, \beta_4, \alpha_4^{-1}, \beta_1^{-1})( \alpha_1^{-1}, \beta_4, \alpha_2^{-1}, \beta_3, \alpha_3, \beta_2, \alpha_2, \beta_3^{-1}).
    \end{align*}

    \begin{figure}[htbp]
        \centering
        \begin{tikzpicture}[xscale=.7,yscale=.7]
            \draw [-<-=.98, -<-=.93, -<-=.865, -<-=.785, -<-=.74] (0,1)--(6,1)--(6,10)--(0,10)--cycle;
            
            \draw (.32,9.1) node {\tiny$\alpha_1$}  (.32,3.5) node {\tiny$\alpha_2$} (.32,5.4) node {\tiny$\alpha_3$} (.33,7.5) node {\tiny$\alpha_4$};
    
            \draw [->-=.2, ->-=.75] (-1.5,8)--(1.5,8);
            \draw (1,8.3) node {\tiny $\beta_1$};

            \draw [-<-=.2, -<-=.75] (-1.5,2)--(1.5,2);
            \draw [-<-=.2, -<-=.75] (-1.5,4)--(1.5,4);
            \draw [->-=.2, ->-=.75] (-1.5,6)--(1.5,6);

            \draw (-1,2.3) node {\tiny $\beta_3$} (-1,4.3) node {\tiny $\beta_4$} (1,6.3) node {\tiny $\beta_2$};

            \draw (-1.5,8)--(-3,8)--(-3,4)--(-1.5,4);
            \draw (1.5,4)--(1.5,3)--(.2,3);
            \draw (-.2,3)--(-1.5,3)--(-1.5,2);
            \draw (1.5,2)--(3,2)--(3,6)--(1.5,6);
            \draw (-1.5,6)--(-1.5,7)--(-.2,7);
            \draw (.2,7)--(1.5,7)--(1.5,8);

            \foreach \x in {1,2,...,8}
            {
            \draw [bend left,-<-=.5] ({9.5+2*cos(\x*45+22.5)}, {7.5+2*sin(\x*45+22.5)}) to ({9.5+2*cos((\x+1)*45+22.5)}, {7.5+2*sin((\x+1)*45+22.5)});
            }

            \draw ({9.5+1.9*cos(90)},{7.5+1.9*sin(90)}) node {\tiny$\alpha_1$} ({9.5+2*cos(135)},{7.5+2*sin(135)}) node {\tiny$\beta_1^{-1}$} ({9.5+2*cos(180)},{7.5+2*sin(180)}) node {\tiny$\alpha_4^{-1}$} ({9.5+2*cos(225)},{7.5+2*sin(225)}) node {\tiny$\beta_4$} ({9.5+2*cos(270)},{7.5+2*sin(270)}) node {\tiny$\alpha_3^{-1}$} ({9.5+2*cos(315)},{7.5+2*sin(315)}) node {\tiny$\beta_1$} ({9.5+2*cos(0)},{7.5+2*sin(0)}) node {\tiny$\alpha_4$} ({9.5+2*cos(45)},{7.5+2*sin(45)}) node {\tiny$\beta_2^{-1}$};

            \foreach \x in {1,2,...,8}
            {
            \draw [bend left,-<-=.5] ({9.5+2*cos(\x*45+22.5)}, {3+2*sin(\x*45+22.5)}) to ({9.5+2*cos((\x+1)*45+22.5)}, {3+2*sin((\x+1)*45+22.5)});
            }

            \draw  ({9.5+1.9*cos(90)},{3+1.9*sin(90)}) node {\tiny$\alpha_1^{-1}$} ({9.5+2*cos(135)},{3+2*sin(135)}) node {\tiny$\beta_3^{-1}$} ({9.5+2*cos(180)},{3+2*sin(180)}) node {\tiny$\alpha_2$} ({9.5+2*cos(225)},{3+2*sin(225)}) node {\tiny$\beta_2$} ({9.5+2*cos(270)},{3+2*sin(270)}) node {\tiny$\alpha_3$} ({9.5+2*cos(315)},{3+2*sin(315)}) node {\tiny$\beta_3$} ({9.5+2*cos(0)},{3+2*sin(0)}) node {\tiny$\alpha_2^{-1}$} ({9.5+2*cos(45)},{3+2*sin(45)}) node {\tiny$\beta_4^{-1}$};
        \end{tikzpicture}
        \caption{Case 4}
        \label{fig: fatg case 4}
    \end{figure}

    \item The permutations $\sigma_0$ and $\sigma_0^{-1}\sigma_1$ are given by (see Figure \ref{fig: fatg case5})
    \begin{align*}
        \sigma_0 & = (\alpha_1, \beta_4^{-1}, \alpha_4^{-1}, \beta_1) (\alpha_2, \beta_3^{-1}, \alpha_1^{-1}, \beta_4) (\alpha_3, \beta_2, \alpha_2^{-1}, \beta_1^{-1}) (\alpha_4, \beta_3, \alpha_3^{-1}, \beta_2^{-1}),\\
        \sigma_0\sigma_1 & = (\alpha_1, \beta_4, \alpha_4^{-1}, \beta_3, \alpha_1^{-1}, \beta_4^{-1}, \alpha_2, \beta_1)( \alpha_4, \beta_1, \alpha_3, \beta_2^{-1}, \alpha_2^{-1}, \beta_3^{-1}, \alpha_3^{-1}, \beta_2).
    \end{align*}

    \begin{figure}[htbp]
        \centering
        \begin{tikzpicture}[xscale=.7,yscale=.7]
            \draw [-<-=.98, -<-=.93, -<-=.865, -<-=.785, -<-=.74] (0,1)--(6,1)--(6,10)--(0,10)--cycle;
            
            \draw (.32,9.1) node {\tiny$\alpha_1$}  (.32,3.5) node {\tiny$\alpha_2$} (.32,5.4) node {\tiny$\alpha_3$} (.33,7.5) node {\tiny$\alpha_4$};
            
            \draw [->-=.2, ->-=.75] (-1.5,8)--(1.5,8);
            \draw (1,8.3) node {\tiny$\beta_1$};

            \draw [->-=.2, ->-=.75] (-1.5,2)--(1.5,2);
            \draw [-<-=.2, -<-=.75] (-1.5,4)--(1.5,4);
            \draw [-<-=.2, -<-=.75] (-1.5,6)--(1.5,6);

            \draw (1,2.3) node {\tiny$\beta_4$} (-1,4.3) node {\tiny$\beta_2$} (-1,6.3) node {\tiny$\beta_3$};

            \draw (-1.5,8)--(-3,8)--(-3,3)--(-.2,3);
            \draw (.2,3)--(1.5,3)--(1.5,2);
            \draw (-1.5,2)--(-2.5,2)--(-2.5,2.8);
            \draw (-2.5,3.2)--(-2.5,6)--(-1.5,6);
            \draw (1.5,6)--(1.5,5)--(.2,5);
            \draw (-.2,5)--(-1.5,5)--(-1.5,4);
            \draw (1.5,4)--(3,4)--(3,8)--(1.5,8);

            \foreach \x in {1,2,...,8}
            {
            \draw [bend left,-<-=.5] ({9.5+2*cos(\x*45+22.5)}, {7.5+2*sin(\x*45+22.5)}) to ({9.5+2*cos((\x+1)*45+22.5)}, {7.5+2*sin((\x+1)*45+22.5)});
            }

            \draw ({9.5+1.9*cos(90)},{7.5+1.9*sin(90)}) node {\tiny$\alpha_1$} ({9.5+2*cos(135)},{7.5+2*sin(135)}) node {\tiny$\beta_1$} ({9.5+2*cos(180)},{7.5+2*sin(180)}) node {\tiny$\alpha_2$} ({9.5+2*cos(225)},{7.5+2*sin(225)}) node {\tiny$\beta_4^{-1}$} ({9.5+2*cos(270)},{7.5+2*sin(270)}) node {\tiny$\alpha_1^{-1}$} ({9.5+2*cos(315)},{7.5+2*sin(315)}) node {\tiny$\beta_3$} ({9.5+2*cos(0)},{7.5+2*sin(0)}) node {\tiny$\alpha_4^{-1}$} ({9.5+2*cos(45)},{7.5+2*sin(45)}) node {\tiny$\beta_4$};

            \foreach \x in {1,2,...,8}
            {
            \draw [bend left,-<-=.5] ({9.5+2*cos(\x*45+22.5)}, {3+2*sin(\x*45+22.5)}) to ({9.5+2*cos((\x+1)*45+22.5)}, {3+2*sin((\x+1)*45+22.5)});
            }

            \draw  ({9.5+1.9*cos(90)},{3+1.9*sin(90)}) node {\tiny$\alpha_4^{-1}$} ({9.5+2*cos(135)},{3+2*sin(135)}) node {\tiny$\beta_2$} ({9.5+2*cos(180)},{3+2*sin(180)}) node {\tiny$\alpha_3^{-1}$} ({9.5+2*cos(225)},{3+2*sin(225)}) node {\tiny$\beta_3^{-1}$} ({9.5+2*cos(270)},{3+2*sin(270)}) node {\tiny$\alpha_2^{-1}$} ({9.5+2*cos(315)},{3+2*sin(315)}) node {\tiny$\beta_2^{-1}$} ({9.5+2*cos(0)},{3+2*sin(0)}) node {\tiny$\alpha_3$} ({9.5+2*cos(45)},{3+2*sin(45)}) node {\tiny$\beta_1$};
        \end{tikzpicture}
        \caption{Case 5}
        \label{fig: fatg case5}
    \end{figure}
    
    \item The permutations $\sigma_0$ and $\sigma_0^{-1}\sigma_1$ are given by (see Figure \ref{fig: fatg case6})
    \begin{align*}
        \sigma_0 & = (\alpha_1, \beta_4^{-1}, \alpha_4^{-1}, \beta_1) (\alpha_2, \beta_3^{-1}, \alpha_1^{-1}, \beta_4) (\alpha_3, \beta_1^{-1}, \alpha_2^{-1}, \beta_2) (\alpha_4, \beta_2^{-1}, \alpha_3^{-1}, \beta_3),\\
        \sigma_0\sigma_1 & = (\alpha_1, \beta_4, \alpha_4^{-1}, \beta_2^{-1}, \alpha_3, \beta_3, \alpha_1^{-1}, \beta_4^{-1}, \alpha_2, \beta_2, \alpha_3^{-1}, \beta_1^{-1}) (\alpha_4, \beta_1, \alpha_2^{-1}, \beta_3^{-1}).
    \end{align*}

    \begin{figure}[htbp]
        \centering
        \begin{tikzpicture}[xscale=.7,yscale=.7]
            \draw [-<-=.98, -<-=.93, -<-=.865, -<-=.785, -<-=.74] (0,1)--(4,1)--(4,10)--(0,10)--cycle;
            
            \draw (.32,9.4) node {\tiny$\alpha_1$}  (.32,3.5) node {\tiny$\alpha_2$} (.32,5.4) node {\tiny$\alpha_3$} (.33,7.5) node {\tiny$\alpha_4$};
    
            \draw [->-=.2, ->-=.75] (-1.5,8)--(1.5,8);
            \draw [->-=.2, ->-=.75] (-1.5,2)--(1.5,2);
            \draw [->-=.2, ->-=.75] (-1.5,4)--(1.5,4);
            \draw [->-=.2, ->-=.75] (-1.5,6)--(1.5,6);

            \draw (1,8.3) node {\tiny $\beta_1$} (1,2.3) node {\tiny $\beta_4$} (1,4.3) node {\tiny $\beta_2$} (1,6.3) node {\tiny $\beta_3$};
             \draw (-1.5,8)--(-3,8)--(-3,3)--(-.2,3);
             \draw (.2,3)--(1.5,3)--(1.5,2);
             \draw (-1.5,2)--(-2.5,2)--(-2.5,2.8);
             \draw (-2.5,3.2)--(-2.5,5)--(-.2,5);
             \draw (.2,5)--(1.5,5)--(1.5,6);
             \draw (-1.5,6)--(-1.5,7)--(-.2,7);
             \draw (.2,7)--(3,7)--(3,4)--(1.5,4);
             \draw (-1.5,4)--(-2,4)--(-2,9)--(-.2,9);
             \draw (.2,9)--(1.5,9)--(1.5,8);
    
             \foreach \x in {1,2,...,12}
             {
             \draw [bend left, -<-=.5] ({7.5+2.5*cos(15+\x*30)}, {7+2.5*sin(15+\x*30)}) to ({7.5+2.5*cos(15+(\x+1)*30)}, {7+2.5*sin(15+(\x+1)*30)});
             }
    
             \draw ({7.5+2.6*cos(0)}, {7+2.6*sin(0)}) node {\tiny$\alpha_1$} ({7.5+2.7*cos(30)}, {7+2.7*sin(30)}) node {\tiny$\beta_1^{-1}$} ({7.5+2.6*cos(60)}, {7+2.6*sin(60)}) node {\tiny$\alpha_3^{-1}$} ({7.5+2.6*cos(90)}, {7+2.6*sin(90)}) node {\tiny$\beta_2$} ({7.5+2.6*cos(120)}, {7+2.6*sin(120)}) node {\tiny$\alpha_2$} ({7.5+2.7*cos(150)}, {7+2.7*sin(150)}) node {\tiny$\beta_4^{-1}$} ({7.5+2.6*cos(180)}, {7+2.6*sin(180)}) node {\tiny$\alpha_1^{-1}$} ({7.5+2.6*cos(210)}, {7+2.6*sin(210)}) node {\tiny$\beta_3$} ({7.5+2.6*cos(240)}, {7+2.6*sin(240)}) node {\tiny$\alpha_3$} ({7.5+2.7*cos(270)}, {7+2.7*sin(270)}) node {\tiny$\beta_2^{-1}$} ({7.5+2.6*cos(300)}, {7+2.6*sin(300)}) node {\tiny$\alpha_4^{-1}$} ({7.5+2.6*cos(330)}, {7+2.6*sin(330)}) node {\tiny$\beta_4$};
    
             \draw [-<-=.125, -<-=.375, -<-=.625, -<-=.875] (6.5,1.5)--(8.5,1.5)--(8.5,3.5)--(6.5,3.5)--cycle;
             \draw (7.3,3.75) node {\tiny$\alpha_4$} (6,2.5) node {\tiny$\beta_3^{-1}$} (7.7,1.2) node {\tiny$\alpha_2^{-1}$} (8.9,2.5) node {\tiny$\beta_1$} ; 
        \end{tikzpicture}
        \caption{Case 6}
        \label{fig: fatg case6}
    \end{figure}

    \item The permutations $\sigma_0$ and $\sigma_0^{-1}\sigma_1$ are given by (see Figure \ref{fig: fatg case7})
    \begin{align*}
        \sigma_0 & = (\alpha_1, \beta_4^{-1}, \alpha_4^{-1}, \beta_1) (\alpha_2, \beta_2^{-1}, \alpha_1^{-1}, \beta_3) (\alpha_3, \beta_3^{-1}, \alpha_2^{-1}, \beta_4) (\alpha_4, \beta_1^{-1}, \alpha_3^{-1}, \beta_2),\\
        \sigma_0\sigma_1 & = (\alpha_1, \beta_4, \alpha_2^{-1}, \beta_2^{-1}, \alpha_4, \beta_1, \alpha_3^{-1}, \beta_4^{-1}, \alpha_2, \beta_3, \alpha_4^{-1}, \beta_1^{-1}) (\alpha_1^{-1}, \beta_3^{-1}, \alpha_3, \beta_2).
    \end{align*}

    \begin{figure}[htbp]
        \centering
        \begin{tikzpicture}[xscale=.7,yscale=.7]
            \draw [-<-=.98, -<-=.93, -<-=.865, -<-=.785, -<-=.74] (0,1)--(4,1)--(4,10)--(0,10)--cycle;
            
            \draw (.32,9.4) node {\tiny$\alpha_1$}  (.32,3.5) node {\tiny$\alpha_2$} (.32,5.4) node {\tiny$\alpha_3$} (.33,7.5) node {\tiny$\alpha_4$};

            \draw [->-=.2, ->-=.75] (-1.5,8)--(1.5,8);
            \draw [->-=.2, ->-=.75] (-1.5,2)--(1.5,2);
            \draw [->-=.2, ->-=.75] (-1.5,4)--(1.5,4);
            \draw [->-=.2, ->-=.75] (-1.5,6)--(1.5,6);

            \draw (1,8.3) node {\tiny $\beta_1$} (1,2.3) node {\tiny $\beta_3$} (1,4.3) node {\tiny $\beta_4$} (1,6.3) node {\tiny $\beta_2$};
             \draw (-1.5,8)--(-3,8)--(-3,5)--(-.2,5);
             \draw (.2,5)--(1.5,5)--(1.5,4);
             \draw (-1.5,4)--(-1.5,3)--(-.2,3);
             \draw (.2,3)--(1.5,3)--(1.5,2);
             \draw (-1.5,2)--(-1.5,1.5)--(-.2,1.5);
             \draw (.2,1.5)--(3,1.5)--(3,6)--(1.5,6);
             \draw (-1.5,6)--(-1.5,7)--(-.2,7);
             \draw (.2,7)--(1.5,7)--(1.5,8);
    
             \foreach \x in {1,2,...,12}
             {
             \draw [bend left, -<-=.5] ({7.5+2.5*cos(15+\x*30)}, {7+2.5*sin(15+\x*30)}) to ({7.5+2.5*cos(15+(\x+1)*30)}, {7+2.5*sin(15+(\x+1)*30)});
             }
    
             \draw ({7.5+2.6*cos(0)}, {7+2.6*sin(0)}) node {\tiny$\alpha_1$} ({7.5+2.7*cos(30)}, {7+2.7*sin(30)}) node {\tiny$\beta_1^{-1}$} ({7.5+2.6*cos(60)}, {7+2.6*sin(60)}) node {\tiny$\alpha_4^{-1}$} ({7.5+2.6*cos(90)}, {7+2.6*sin(90)}) node {\tiny$\beta_3$} ({7.5+2.6*cos(120)}, {7+2.6*sin(120)}) node {\tiny$\alpha_2$} ({7.5+2.7*cos(150)}, {7+2.7*sin(150)}) node {\tiny$\beta_4^{-1}$} ({7.5+2.6*cos(180)}, {7+2.6*sin(180)}) node {\tiny$\alpha_3^{-1}$} ({7.5+2.6*cos(210)}, {7+2.6*sin(210)}) node {\tiny$\beta_1$} ({7.5+2.6*cos(240)}, {7+2.6*sin(240)}) node {\tiny$\alpha_4$} ({7.5+2.7*cos(270)}, {7+2.7*sin(270)}) node {\tiny$\beta_2^{-1}$} ({7.5+2.6*cos(300)}, {7+2.6*sin(300)}) node {\tiny$\alpha_2^{-1}$} ({7.5+2.6*cos(330)}, {7+2.6*sin(330)}) node {\tiny$\beta_4$};
    
             \draw [-<-=.125, -<-=.375, -<-=.625, -<-=.875] (6.5,1.5)--(8.5,1.5)--(8.5,3.5)--(6.5,3.5)--cycle;
             \draw (7.3,3.75) node {\tiny$\alpha_1^{-1}$} (6,2.5) node {\tiny$\beta_2$} (7.7,1.2) node {\tiny$\alpha_3$} (8.9,2.5) node {\tiny$\beta_3^{-1}$}; 
        \end{tikzpicture}
        \caption{Case 7}
        \label{fig: fatg case7}
    \end{figure}

    \item \label{case7} The permutations $\sigma_0$ and $\sigma_0^{-1}\sigma_1$ are given by (see Figure \ref{fig: fatg case 8})
    \begin{align*}
        \sigma_0 & = (\alpha_1, \beta_4^{-1}, \alpha_4^{-1}, \beta_1) (\alpha_2, \beta_1^{-1}, \alpha_1^{-1}, \beta_2) (\alpha_3, \beta_3^{-1}, \alpha_2^{-1}, \beta_4) (\alpha_4, \beta_2^{-1}, \alpha_3^{-1}, \beta_3),\\
        \sigma_0\sigma_1 & =  (\alpha_1, \beta_2, \alpha_3^{-1}, \beta_3^{-1}, \alpha_4, \beta_1, \alpha_1^{-1}, \beta_4^{-1}, \alpha_3, \beta_3, \alpha_2^{-1}, \beta_1^{-1})(\alpha_2, \beta_4, \alpha_4^{-1}, \beta_2^{-1}).
    \end{align*}

    \begin{figure}[htbp]
        \centering
        \begin{tikzpicture}[xscale=.7,yscale=.7]
            \draw [-<-=.98, -<-=.93, -<-=.865, -<-=.785, -<-=.74] (0,1)--(4,1)--(4,10)--(0,10)--cycle;
            
            \draw (.32,9.4) node {\tiny$\alpha_1$}  (.32,3.5) node {\tiny$\alpha_2$} (.32,5.4) node {\tiny$\alpha_3$} (.33,7.5) node {\tiny$\alpha_4$};
            
            \draw [->-=.2, ->-=.75] (-1.5,8)--(1.5,8);
            \draw [->-=.2, ->-=.75] (-1.5,2)--(1.5,2);
            \draw [->-=.2, ->-=.75] (-1.5,4)--(1.5,4);
            \draw [->-=.2, ->-=.75] (-1.5,6)--(1.5,6);

            \draw (1,8.3) node {\tiny $\beta_1$} (1,2.3) node {\tiny $\beta_2$} (1,4.3) node {\tiny $\beta_4$} (1,6.3) node {\tiny $\beta_3$};
             \draw (-1.5,8)--(-3,8)--(-3,5)--(-.2,5);
             \draw (.2,5)--(1.5,5)--(1.5,4);
             \draw (-1.5,4)--(-1.5,3)--(-.2,3);
             \draw(.2,3)--(2,3)--(2,6)--(1.5,6);
             \draw (-1.5,6)--(-1.5,7)--(-.2,7);
             \draw (.2,7)--(3,7)--(3,2)--(1.5,2);
             \draw (-1.5,2)--(-2,2)--(-2,4.8);
             \draw (-2,5.2)--(-2,7.8);
             \draw (-2,8.2)--(-2,9)--(-.2,9);
             \draw (.2,9)--(1.5,9)--(1.5,8);
    
             \foreach \x in {1,2,...,12}
             {
             \draw [bend left, -<-=.5] ({7.5+2.5*cos(15+\x*30)}, {7+2.5*sin(15+\x*30)}) to ({7.5+2.5*cos(15+(\x+1)*30)}, {7+2.5*sin(15+(\x+1)*30)});
             }
    
             \draw ({7.5+2.6*cos(0)}, {7+2.6*sin(0)}) node {\tiny$\alpha_1$} ({7.5+2.7*cos(30)}, {7+2.7*sin(30)}) node {\tiny$\beta_1^{-1}$} ({7.5+2.6*cos(60)}, {7+2.6*sin(60)}) node {\tiny$\alpha_2^{-1}$} ({7.5+2.6*cos(90)}, {7+2.6*sin(90)}) node {\tiny$\beta_3$} ({7.5+2.6*cos(120)}, {7+2.6*sin(120)}) node {\tiny$\alpha_3$} ({7.5+2.7*cos(150)}, {7+2.7*sin(150)}) node {\tiny$\beta_4^{-1}$} ({7.5+2.6*cos(180)}, {7+2.6*sin(180)}) node {\tiny$\alpha_1^{-1}$} ({7.5+2.6*cos(210)}, {7+2.6*sin(210)}) node {\tiny$\beta_1$} ({7.5+2.6*cos(240)}, {7+2.6*sin(240)}) node {\tiny$\alpha_4$} ({7.5+2.7*cos(270)}, {7+2.7*sin(270)}) node {\tiny$\beta_3^{-1}$} ({7.5+2.6*cos(300)}, {7+2.6*sin(300)}) node {\tiny$\alpha_3^{-1}$} ({7.5+2.6*cos(330)}, {7+2.6*sin(330)}) node {\tiny$\beta_2$};
    
             \draw [-<-=.125, -<-=.375, -<-=.625, -<-=.875] (6.5,1.5)--(8.5,1.5)--(8.5,3.5)--(6.5,3.5)--cycle;
             \draw (7.3,3.75) node {\tiny$\alpha_2$} (6,2.5) node {\tiny$\beta_2^{-1}$} (7.7,1.2) node {\tiny$\alpha_4^{-1}$} (8.9,2.5) node {\tiny$\beta_4$}; 
        \end{tikzpicture}
        \caption{Case 8}
        \label{fig: fatg case 8}
    \end{figure}
    \end{enumerate}
    One can see that this is a complete list of fat graphs with the required conditions. This completes the proof of the proposition.
\end{proof}


\begin{note}\label{note2.3}
    The boundary words for cases of type $\{4,12\}$ are the same modulo relabeling of the edges. For example, if in Case \ref{case1}, we relabel the edges $\alpha_i$ by $\alpha_{i+1}$ (modulo 4) and $\beta_i$ by $\beta_{i+1}$ (modulo 4), then the boundary word is the same as in Case \ref{case7}. We can check that the same holds for filling pairs of type $\{8,8\}$. 
\end{note}

The observation made in Note \ref{note2.3} yields the following.
\begin{theorem}
    There are exactly two (up to mapping class group action) minimal filling pairs in $S_2$.
\end{theorem}
\begin{proof}
    It follows from Proposition~\ref{prop1.1} that there are at least two minimal filling pairs on $S_2$, namely one of type $\{4,12\}$ and one of type $\{8,8\}$. We show that there exist unique filling pairs, each of types $\{4,12\}$ and $\{8,8\}$, up to homeomorphism. 

    We prove the uniqueness of the filling pairs of type $\{4,12\}$. The uniqueness of the other types follows similarly. Let $(\alpha, \beta)$ and $(\alpha', \beta')$ be two pairs of fillings in $S_2$ of type $\{4,12\}$. Suppose $S_2\setminus (\alpha\cup \beta)=D_1 \bigsqcup D_2, S_2\setminus (\alpha'\cup \beta')=D_1' \bigsqcup D_2', |D_1|=|D_1'|=4, |D_2|=|D_2'|= 12$. We label the subarcs of the curves in each of the filling pairs in such a way that the boundary words of the complementary disks are the same modulo prime symbol (this is possible by Note~\ref{note2.3}). Let $q: D_1 \bigsqcup D_2 \to S_2$ and $q': D_1' \bigsqcup D_2' \to S_2$ be the quotient maps and $f: D_1 \bigsqcup D_2 \to D_1' \bigsqcup D_2'$ be a homeomorphism mapping edges of $D_1 \bigsqcup D_2$ to the corresponding edges of $D_1' \bigsqcup D_2'$ (e.g. $\alpha_1 \mapsto \alpha_1'$ and so on). Then the homeomorphism of the surface $S_2$ mapping $(\alpha, \beta)$ to $(\alpha', \beta')$ is given by $q'\circ f \circ q^{-1}$.
\end{proof}

\section{Length of minimal filling pairs in genus two}
In this section, we find a lower bound for the length of minimal filling pairs on a genus-two surface. We do this into two parts. First, we determine the minimal length of filling pairs of type $\{8,8\}$ in the following theorem. We conclude this section by minimizing the length of filling pairs of type $\{4,12\}$ in Theorem \ref{thm 4.3}.

\begin{theorem}\label{thm: length}
    Let $(\alpha, \beta)$ be a filling pair of type $\{8,8\}$ in a hyperbolic surface $X\in \mathcal{T}_2$. Then
    $$\mathrm{length}_X(\alpha,\beta) \geq 8\cdot\cosh^{-1}\left(\sqrt{2}+1\right).$$
\end{theorem}
\begin{proof}
    We assume that both the curves $\alpha$ and $\beta$ are simple closed geodesics. Let the angles at the consecutive intersection points along $\alpha$ be denoted by $\theta_1, \theta_2, \theta_3,$ and $\theta_4$ (see Figure~\ref{fig: 3.1}). By assumption, the filling pair $(\alpha,\beta)$ decompose the surface into two octagons, denoted by, $P_1$ and $P_2$. Then the interior angles of the octagons are as in Figure \ref{fig: 3.1}. Using the Gauss-Bonnet theorem, areas of the polygons are $(2\pi-\theta)$ and $(2\pi+\theta)$, where $\theta=\theta_1+\theta_2-\theta_3-\theta_4$. The value of $\theta$ ranges in $(-2\pi,2\pi)$. Now suppose $\tilde{P_1}$ and $\tilde{P_2}$ are regular octagons with $\mathrm{Area}(\tilde{P_i})=\mathrm{Area}(P_i), i=1,2$. It is well known that among all hyperbolic $n$-gons with fixed area, the regular one gives the least perimeter (see \cite{Bezdek}). Therefore, it follows that
    $$\mathrm{Perim}(\tilde{P_1})+\mathrm{Perim}(\tilde{P_2})\leq \mathrm{Perim}(P_1)+\mathrm{Perim}(P_2).$$

    Using basic hyperbolic trigonometry, we have
    \begin{align*}
        \mathrm{Perim}(\tilde{P_1})&=8 \cdot \cosh^{-1}\left(\frac{\cos\left(\frac{\pi}{4}\right)+\cos^2\left(\frac{\pi}{4}+\frac{\theta}{16}\right)}{\sin^2\left(\frac{\pi}{4}+\frac{\theta}{16}\right)}\right),\\
        \mathrm{Perim}(\tilde{P_2})&=8\cdot \cosh^{-1}\left(\frac{\cos\left(\frac{\pi}{4}\right)+\cos^2\left(\frac{\pi}{4}-\frac{\theta}{16}\right)}{\sin^2\left(\frac{\pi}{4}-\frac{\theta}{16}\right)}\right).
    \end{align*}

    From Lemma \ref{lem 1.2} below, it follows that $2\cdot (8 \cdot \cosh^{-1}(\sqrt{2}+1))\leq \mathrm{Perim}(\tilde{P_1})+\mathrm{Perim}(\tilde{P_2})$, where $8 \cdot \cosh^{-1}(\sqrt{2}+1)$ is the perimeter of a regular right-angled octagon with area $2\pi$. By the Gauss–Bonnet theorem, the interior angle of a regular octagon with area $2\pi$ is $\frac{\pi}{2}$. This shows that the inequality in the thoerem is sharp, providing a tight lower bound for filling pairs of type $\{8,8\}$.
    \begin{figure}[htbp]
        \centering
        \begin{tikzpicture}[xscale=.8,yscale=.8]
            \draw [-<-=.98, -<-=.93, -<-=.865, -<-=.785, -<-=.74] (0,1)--(6,1)--(6,10)--(0,10)--cycle;
            
            \draw (.32,9.1) node {\tiny$\alpha_1$}  (.32,3.5) node {\tiny$\alpha_2$} (.32,5.4) node {\tiny$\alpha_3$} (.33,7.5) node {\tiny$\alpha_4$};
            
            \draw (-.3,8.2) node {\tiny $\theta_1$} (-.3,6.2) node {\tiny $\theta_2$} (-.3,4.2) node {\tiny $\theta_3$} (-.3,2.2) node {\tiny $\theta_4$};
    
            \draw [->-=.2, ->-=.75] (-1.5,8)--(1.5,8);
            \draw (1,8.3) node {\tiny $\beta_1$};

            \draw [-<-=.2, -<-=.75] (-1.5,2)--(1.5,2);
            \draw [-<-=.2, -<-=.75] (-1.5,4)--(1.5,4);
            \draw [->-=.2, ->-=.75] (-1.5,6)--(1.5,6);

            \draw (-1,2.3) node {\tiny $\beta_3$} (-1,4.3) node {\tiny $\beta_2$} (1,6.3) node {\tiny $\beta_4$};
    
            \draw (-1.5,8)--(-1.5,7)--(-.2,7);
            \draw (.2,7)--(1.5,7)--(1.5,6);
            \draw (-1.5,6)--(-3,6)--(-3,2)--(-1.5,2);
            \draw (1.5,2)--(1.5,3)--(.2,3);
            \draw (-.2,3)--(-1.5,3)--(-1.5,4);
            \draw (1.5,4)--(3,4)--(3,8)--(1.5,8);

              \draw [bend left, -<-=.5] ({9.5+2*cos(22.7)},{7.5+2*sin(22.7)}) to ({9.5+2*cos(67.5)},{7.5+2*sin(67.5)});
              \draw [bend left, -<-=.5] ({9.5+2*cos(67.5)},{7.5+2*sin(67.5)}) to ({9.5+2*cos(112.5)},{7.5+2*sin(112.5)});
              \draw [bend left, -<-=.5] ({9.5+2*cos(112.5)},{7.5+2*sin(112.5)}) to ({9.5+2*cos(157.5)},{7.5+2*sin(157.5)});
              \draw [bend left, -<-=.5] ({9.5+2*cos(157.5)},{7.5+2*sin(157.5)}) to ({9.5+2*cos(202.5)},{7.5+2*sin(202.5)});
              \draw [bend left, -<-=.5] ({9.5+2*cos(202.5)},{7.5+2*sin(202.5)}) to ({9.5+2*cos(247.5)}, {7.5+2*sin(247.5)});
              \draw [bend left, -<-=.5] ({9.5+2*cos(247.5)}, {7.5+2*sin(247.5)}) to ({9.5+2*cos(292.5)},{7.5+2*sin(292.5)});
              \draw [bend left, -<-=.5] ({9.5+2*cos(292.5)},{7.5+2*sin(292.5)}) to ({9.5+2*cos(337.5)},{7.5+2*sin(337.5)});
              \draw [bend left, -<-=.5] ({9.5+2*cos(337.5)},{7.5+2*sin(337.5)}) to ({9.5+2*cos(22.7)},{7.5+2*sin(22.7)});
    
              \draw  ({9.5+1.9*cos(90)},{7.5+1.9*sin(90)}) node {\tiny$\alpha_1$} ({9.5+2*cos(135)},{7.5+2*sin(135)}) node {\tiny$\beta_1^{-1}$} ({9.5+2*cos(180)},{7.5+2*sin(180)}) node {\tiny$\alpha_2$} ({9.5+2*cos(225)},{7.5+2*sin(225)}) node {\tiny$\beta_2$} ({9.5+2*cos(270)},{7.5+2*sin(270)}) node {\tiny$\alpha_3^{-1}$} ({9.5+2*cos(315)},{7.5+2*sin(315)}) node {\tiny$\beta_3$} ({9.5+2*cos(0)},{7.5+2*sin(0)}) node {\tiny$\alpha_2^{-1}$} ({9.5+2*cos(45)},{7.5+2*sin(45)}) node {\tiny$\beta_2^{-1}$};

              \draw [bend left, -<-=.5] ({9.5+2*cos(22.7)},{3+2*sin(22.7)}) to ({9.5+2*cos(67.5)},{3+2*sin(67.5)});
              \draw [bend left, -<-=.5] ({9.5+2*cos(67.5)},{3+2*sin(67.5)}) to ({9.5+2*cos(112.5)},{3+2*sin(112.5)});
              \draw [bend left, -<-=.5] ({9.5+2*cos(112.5)},{3+2*sin(112.5)}) to ({9.5+2*cos(157.5)},{3+2*sin(157.5)});
              \draw [bend left, -<-=.5] ({9.5+2*cos(157.5)},{3+2*sin(157.5)}) to ({9.5+2*cos(202.5)},{3+2*sin(202.5)});
              \draw [bend left, -<-=.5] ({9.5+2*cos(202.5)},{3+2*sin(202.5)}) to ({9.5+2*cos(247.5)}, {3+2*sin(247.5)});
              \draw [bend left, -<-=.5] ({9.5+2*cos(247.5)}, {3+2*sin(247.5)}) to ({9.5+2*cos(292.5)},{3+2*sin(292.5)});
              \draw [bend left, -<-=.5] ({9.5+2*cos(292.5)},{3+2*sin(292.5)}) to ({9.5+2*cos(337.5)},{3+2*sin(337.5)});
              \draw [bend left, -<-=.5] ({9.5+2*cos(337.5)},{3+2*sin(337.5)}) to ({9.5+2*cos(22.7)},{3+2*sin(22.7)});
    
              \draw  ({9.5+1.9*cos(90)},{3+1.9*sin(90)}) node {\tiny$\alpha_1^{-1}$} ({9.5+2*cos(135)},{3+2*sin(135)}) node {\tiny$\beta_3^{-1}$} ({9.5+2*cos(180)},{3+2*sin(180)}) node {\tiny$\alpha_4^{-1}$} ({9.5+2*cos(225)},{3+2*sin(225)}) node {\tiny$\beta_4$} ({9.5+2*cos(270)},{3+2*sin(270)}) node {\tiny$\alpha_3$} ({9.5+2*cos(315)},{3+2*sin(315)}) node {\tiny$\beta_1$} ({9.5+2*cos(0)},{3+2*sin(0)}) node {\tiny$\alpha_4$} ({9.5+2*cos(45)},{3+2*sin(45)}) node {\tiny$\beta_4^{-1}$};

               \draw ({9.5+1.6*cos(67.5)},{7.5+1.6*sin(67.5)}) node {\tiny $\theta_4$} ({9.5+1.6*cos(112.5)},{7.5+1.6*sin(112.5)}) node {\tiny $\overline{\theta}_1$} ({9.5+1.6*cos(157.5)},{7.5+1.6*sin(157.5)}) node {\tiny $\theta_3$} ({9.5+1.6*cos(202.5)},{7.5+1.6*sin(202.5)}) node {\tiny $\overline{\theta}_4$} ({9.5+1.6*cos(247.5)},{7.5+1.6*sin(247.5)}) node {\tiny $\theta_3$} ({9.5+1.6*cos(292.5)},{7.5+1.6*sin(292.5)}) node {\tiny $\overline{\theta}_2$} ({9.5+1.6*cos(337.5)},{7.5+1.6*sin(337.5)}) node {\tiny $\theta_4$} ({9.5+1.6*cos(22.5)},{7.5+1.6*sin(22.5)}) node {\tiny $\overline{\theta}_3$};

               \draw ({9.5+1.6*cos(67.5)},{3+1.6*sin(67.5)}) node {\tiny $\theta_1$} ({9.5+1.6*cos(112.5)},{3+1.6*sin(112.5)}) node {\tiny $\overline{\theta}_4$} ({9.5+1.6*cos(157.5)},{3+1.6*sin(157.5)}) node {\tiny $\theta_2$} ({9.5+1.6*cos(202.5)},{3+1.6*sin(202.5)}) node {\tiny $\overline{\theta}_1$} ({9.5+1.6*cos(247.5)},{3+1.6*sin(247.5)}) node {\tiny $\theta_2$} ({9.5+1.6*cos(292.5)},{3+1.6*sin(292.5)}) node {\tiny $\overline{\theta}_3$} ({9.5+1.6*cos(337.5)},{3+1.6*sin(337.5)}) node {\tiny $\theta_1$} ({9.5+1.6*cos(22.5)},{3+1.6*sin(22.5)}) node {\tiny $\overline{\theta}_2$};

               \draw (9.5,3) node{$P_1$} (9.5,7.5) node{$P_2$};
        \end{tikzpicture}
        \caption{}
        \label{fig: 3.1}
    \end{figure}
\end{proof}

\begin{lemma}\label{lem 1.2}
    Suppose $\tilde{P_1}$ and $\tilde{P_2}$ are as defined in the proof of Theorem \ref{thm: length}. Then 
    $$\mathrm{Perim}(\tilde{P_1})+\mathrm{Perim}(\tilde{P_2}) \geq 2\cdot (8\cosh^{-1}(\sqrt{2}+1)).$$
\end{lemma}
\begin{proof}
    Let
    \begin{align*}
        f(\theta)& \coloneqq \mathrm{Perim}(\tilde{P_1})+\mathrm{Perim}(\tilde{P_2})\\
    &=8 \cdot \cosh^{-1}\left(\frac{\cos\left(\frac{\pi}{4}\right)+\cos^2\left(\frac{\pi}{4}+\frac{\theta}{16}\right)}{\sin^2\left(\frac{\pi}{4}+\frac{\theta}{16}\right)}\right) + 8\cdot \cosh^{-1}\left(\frac{\cos\left(\frac{\pi}{4}\right)+\cos^2\left(\frac{\pi}{4}-\frac{\theta}{16}\right)}{\sin^2\left(\frac{\pi}{4}-\frac{\theta}{16}\right)}\right)\\
    &=8 \cdot \cosh^{-1} \left ({\frac{\sqrt{2}+1-\sin{\left (\frac{\theta}{2}\right)}}{1+\sin{\left (\frac{\theta}{2}\right)}}} \right ) + 8 \cdot \cosh^{-1}\left ({\frac{\sqrt{2}+1 + \sin{\left (\frac{\theta}{2}\right)}}{1 - \sin{\left (\frac{\theta}{2}\right)}}}\right).
    \end{align*}
    Then 
    $$f'(\theta) = \frac{(2 + \sqrt{2}) \cos{\left (\frac{\theta}{8}\right )}}{ \sqrt{2+2 \sqrt{2}
    }} \cdot \left[ \frac{1}{\sqrt{1+\sqrt{2} \sin {\left (\frac{\theta}{8}\right )}}} - \frac{1}{\sqrt{1 - \sqrt{2} \sin {\left (\frac{\theta}{8}\right )}}} \right].$$
    It is straightforward to see that
    \begin{align*}
        f'(\theta)&<0, \text{ for } -2\pi<\theta<0,\\
        &=0 \text{ for } \theta =0,\\
        &>0 \text{ for } 0<\theta <2 \pi.
    \end{align*}
    This shows that the function $f(\theta)$ has a minimum at $\theta =0$. This completes the proof of the lemma.
\end{proof}

\subsection{The length of \texorpdfstring{$\{4,12\}$}{} filling pairs}\label{sub 4.1}
The aim of this subsection is to prove the following theorem.

\begin{theorem}\label{thm 4.3}
    Let $(\alpha, \beta)$ be a $\{4,12\}$ filling pair in a hyperbolic surface $X\in \mathcal{T}_2$. Then
    $$\mathrm{length}_X(\alpha,\beta) \geq L_0,$$
    where $L_0$ is $6 \cdot \cosh^{-1}(7/2)\approx 11.5490838.$
\end{theorem}
\begin{proof}
    Let the angle between $\alpha$ and $\beta$ at the intersection points be $\phi_1, \phi_2, \phi_3$ and $\phi_4$ (listed along $\alpha$). Let $X\setminus (\alpha \cup \beta)=Q_1 \bigsqcup Q_2$, where $Q_1$ and $Q_2$ are 12-gon and 4-gon, respectively (see Figure \ref{fig: 1.2}). Then the areas of $Q_1$ and $Q_2$ are given by $4\pi -\phi$ and $\phi$, where $\phi = \phi_1 + \phi_2- \phi_3- \phi_4$. The value of $\phi$ ranges in the interval $(0,2\pi)$.

\begin{figure}[htbp]
    \centering
    \begin{tikzpicture}[xscale=.8,yscale=.8]
        \draw [-<-=.68, -<-=.93, -<-=.865, -<-=.785, -<-=.74] (0,1)--(4,1)--(4,10)--(0,10)--cycle;
            
            \draw (.32,9.4) node {\tiny$\alpha_1$}  (.32,3.5) node {\tiny$\alpha_2$} (.32,5.4) node {\tiny$\alpha_3$} (.33,7.5) node {\tiny$\alpha_4$};
    
            \draw [->-=.2, ->-=.75] (-1.5,8)--(1.5,8);
            \draw [->-=.2, ->-=.75] (-1.5,2)--(1.5,2);
            \draw [->-=.2, ->-=.75] (-1.5,4)--(1.5,4);
            \draw [->-=.2, ->-=.75] (-1.5,6)--(1.5,6);
    
            \draw (1,8.3) node { \tiny$\beta_1$} (1,2.3) node { \tiny$\beta_3$} (1,4.3) node { \tiny$\beta_2$} (1,6.3) node {\tiny$\beta_4$};
             \draw (-1.5,8)--(-1.5,7)--(-.2,7);
             \draw (.2,7)--(1.5,7)--(1.5,6)--(-4.5,6)--(-4.5,3)--(-.2,3);
             \draw (.2,3)--(1.5,3)--(1.5,2)--(-3,2)--(-3,2.8);
             \draw (-3,3.2)--(-3,5)--(-.2,5);
             \draw (.2,5)--(1.5,5)--(1.5,4)--(-2,4)--(-2,4.8);
             \draw (-2,5.2)--(-2,5.8);
             \draw (-2,6.2)--(-2,9);
             \draw (-2,9)--(-.2,9);
             \draw (.2,9)--(1.5,9)--(1.5,8);

             \draw (-.3,8.2) node {\tiny $\phi_1$} (-.3,6.2) node {\tiny $\phi_2$} (-.3,4.2) node {\tiny $\phi_3$} (-.3,2.2) node {\tiny $\phi_4$};
    
             \foreach \x in {1,2,...,12}
             {
             \draw [bend left, -<-=.5] ({7.5+2.5*cos(15+\x*30)}, {7+2.5*sin(15+\x*30)}) to ({7.5+2.5*cos(15+(\x+1)*30)}, {7+2.5*sin(15+(\x+1)*30)});
             }
    
             \draw ({7.5+2.6*cos(0)}, {7+2.6*sin(0)}) node {\tiny$\alpha_3$} ({7.5+2.7*cos(30)}, {7+2.7*sin(30)}) node {\tiny$\beta_2^{-1}$} ({7.5+2.6*cos(60)}, {7+2.6*sin(60)}) node {\tiny$\alpha_2^{-1}$} ({7.5+2.6*cos(90)}, {7+2.6*sin(90)}) node {\tiny$\beta_1$} ({7.5+2.6*cos(120)}, {7+2.6*sin(120)}) node {\tiny$\alpha_4$} ({7.5+2.7*cos(150)}, {7+2.7*sin(150)}) node {\tiny$\beta_4^{-1}$} ({7.5+2.6*cos(180)}, {7+2.6*sin(180)}) node {\tiny$\alpha_1^{-1}$} ({7.5+2.6*cos(210)}, {7+2.6*sin(210)}) node {\tiny$\beta_2$} ({7.5+2.6*cos(240)}, {7+2.6*sin(240)}) node {\tiny$\alpha_2$} ({7.5+2.7*cos(270)}, {7+2.7*sin(270)}) node {\tiny$\beta_3^{-1}$} ({7.5+2.6*cos(300)}, {7+2.6*sin(300)}) node {\tiny$\alpha_4^{-1}$} ({7.5+2.6*cos(330)}, {7+2.6*sin(330)}) node {\tiny$\beta_4$};
    
             \draw [-<-=.125, -<-=.375, -<-=.625, -<-=.875] (6.5,1.5)--(8.5,1.5)--(8.5,3.5)--(6.5,3.5)--cycle;
             \draw (7.3,3.75) node {\tiny$\alpha_1$} (7.7,1.2) node {\tiny$\alpha_3^{-1}$} (8.9,2.5) node {\tiny$\beta_3$} (6,2.5) node {\tiny$\beta_1^{-1}$};

         \draw ({7.5+2.2*cos(15))}, {7+2.2*sin(15)}) node {\tiny $\overline{\phi}_3$} ({7.5+2.2*cos(45))}, {7+2.2*sin(45)}) node {\tiny $\phi_4$} ({7.5+2.2*cos(75))}, {7+2.2*sin(75)}) node {\tiny $\overline{\phi}_3$} ({7.5+2.2*cos(105))}, {7+2.2*sin(105)}) node {\tiny $\phi_1$} ({7.5+2.2*cos(135))}, {7+2.2*sin(135)}) node {\tiny $\overline{\phi}_2$} ({7.5+2.2*cos(165))}, {7+2.2*sin(165)}) node {\tiny $\phi_1$} ({7.5+2.2*cos(195))}, {7+2.2*sin(195)}) node {\tiny $\overline{\phi}_4$} ({7.5+2.2*cos(225))}, {7+2.2*sin(225)}) node {\tiny $\phi_3$} ({7.5+2.2*cos(255))}, {7+2.2*sin(255)}) node {\tiny $\overline{\phi}_4$} ({7.5+2.2*cos(285))}, {7+2.2*sin(285)}) node {\tiny $\phi_2$} ({7.5+2.2*cos(315))}, {7+2.2*sin(315)}) node {\tiny $\overline{\phi}_1$} ({7.5+2.2*cos(345))}, {7+2.2*sin(345)}) node {\tiny $\phi_2$};

         \draw (8.2,3.2) node {\tiny $\phi_4$} (8.2,1.8) node {\tiny $\overline{\phi}_2$} (6.8,1.8) node {\tiny $\phi_3$} (6.8,3.2) node {\tiny $\overline{\phi}_1$};

         \draw (7.5, 7) node {$Q_1$} (7.5,2.5) node {$Q_2$};
    \end{tikzpicture}
    \caption{}
    \label{fig: 1.2}
\end{figure}

Now we investigate the properties of the surface $X_0 \in \mathcal{T}_2$ where the filling pair $(\alpha, \beta)$ has minimum length. The existence and uniqueness of such a surface $X_0$ is guaranteed by Proposition 2.4 \cite{Schmutz1993}. It is well known that if $\tau: S_2 \to S_2$ is a diffeomorphism, the minimum point for the length of  $(\tau(\alpha), \tau(\beta))$ is $\tau^{-1}\cdot X_0$ (see Chapter 12 \cite{MR2850125}). Let $\tau_1$ and $\tau_2$ be the reflections about the red and green axes, respectively, as shown in Figure \ref{fig: reflections}. It is clear that $\tau_i(\alpha)= \alpha$ and $\tau_i(\beta)= \beta$ and hence $\tau_i\cdot X_0 =X_0, i=1,2$. As $\tau_1(\alpha_4)= \alpha_2$, we get $l_{X_0}(\alpha_4)= l_{\tau_1 \cdot X_0}(\alpha_2)= l_{X_0}(\alpha_2)$. Similarly, we can show that $l_{X_0}(\beta_2)= l_{X_0}(\beta_4), l_{X_0}(\alpha_1)= l_{X_0}(\alpha_3)$. We also have the following conditions on the angles: $\phi_2= \overline{\phi}_4$ and $\phi_3= \overline{\phi}_1$. Now, using the map $\tau_2$, we have $l_{X_0}(\alpha_3)= l_{X_0}(\beta_3), l_{X_0}(\alpha_1)= l_{X_0}(\beta_1)$ and $\phi_3 = \phi_4$. Combining the above relations, we have $\phi_1= \phi_2= \overline{\phi}_3= \overline{\phi}_4= \theta \text{ (say}), l_{X_0}(\alpha_1)= l_{X_0}(\beta_1) = l_{X_0}(\alpha_3)= l_{X_0}(\beta_3)$ and $ l_{X_0}(\alpha_2)= l_{X_0}(\beta_2)=l_{X_0}(\alpha_4)= l_{X_0}(\beta_4)$.

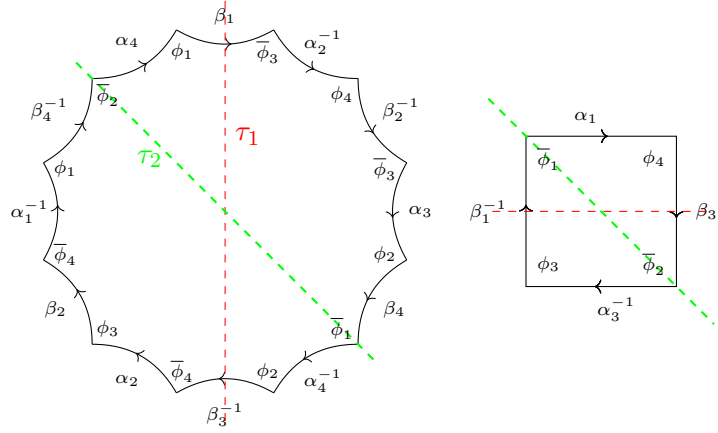
\begin{figure}[htbp]
    \centering
    \begin{tikzpicture}
        \foreach \x in {1,2,...,12}
             {
             \draw [bend left, -<-=.5] ({7.5+2.5*cos(15+\x*30)}, {7+2.5*sin(15+\x*30)}) to ({7.5+2.5*cos(15+(\x+1)*30)}, {7+2.5*sin(15+(\x+1)*30)});
             }
    
             \draw ({7.5+2.6*cos(0)}, {7+2.6*sin(0)}) node {\tiny$\alpha_3$} ({7.5+2.7*cos(30)}, {7+2.7*sin(30)}) node {\tiny$\beta_2^{-1}$} ({7.5+2.6*cos(60)}, {7+2.6*sin(60)}) node {\tiny$\alpha_2^{-1}$} ({7.5+2.6*cos(90)}, {7+2.6*sin(90)}) node {\tiny$\beta_1$} ({7.5+2.6*cos(120)}, {7+2.6*sin(120)}) node {\tiny$\alpha_4$} ({7.5+2.7*cos(150)}, {7+2.7*sin(150)}) node {\tiny$\beta_4^{-1}$} ({7.5+2.6*cos(180)}, {7+2.6*sin(180)}) node {\tiny$\alpha_1^{-1}$} ({7.5+2.6*cos(210)}, {7+2.6*sin(210)}) node {\tiny$\beta_2$} ({7.5+2.6*cos(240)}, {7+2.6*sin(240)}) node {\tiny$\alpha_2$} ({7.5+2.7*cos(270)}, {7+2.7*sin(270)}) node {\tiny$\beta_3^{-1}$} ({7.5+2.6*cos(300)}, {7+2.6*sin(300)}) node {\tiny$\alpha_4^{-1}$} ({7.5+2.6*cos(330)}, {7+2.6*sin(330)}) node {\tiny$\beta_4$};
    
             \draw [-<-=.125, -<-=.375, -<-=.625, -<-=.875] (11.5,6)--(13.5,6)--(13.5,8)--(11.5,8)--cycle;
             \draw (5+7.3,3.75+4.5) node {\tiny$\alpha_1$} (5+7.7,1.2+4.5) node {\tiny$\alpha_3^{-1}$} (5+8.9,2.5+4.5) node {\tiny$\beta_3$} (5+6,2.5+4.5) node {\tiny$\beta_1^{-1}$};

             \draw [dashed, red] ({7.5+2.8*cos(90)}, {7+2.8*sin(90)}) to ({7.5+2.8*cos(270)}, {7+2.8*sin(270)});

             \draw [dashed, red] (13.9,7) to (11,7);

             \draw[dashed, thick, green] ({7.5+2.8*cos(135)}, {7+2.8*sin(135)}) to ({7.5+2.8*cos(315)}, {7+2.8*sin(315)});
             \draw[dashed, thick, green] (11,8.5) to (14,5.5);

             \draw [red] (7.8,8) node{$\tau_1$};
             \draw [green] (6.5,7.7) node{$\tau_2$};

             \draw ({7.5+2.2*cos(15))}, {7+2.2*sin(15)}) node {\tiny $\overline{\phi}_3$} ({7.5+2.2*cos(45))}, {7+2.2*sin(45)}) node {\tiny $\phi_4$} ({7.5+2.2*cos(75))}, {7+2.2*sin(75)}) node {\tiny $\overline{\phi}_3$} ({7.5+2.2*cos(105))}, {7+2.2*sin(105)}) node {\tiny $\phi_1$} ({7.5+2.2*cos(135))}, {7+2.2*sin(135)}) node {\tiny $\overline{\phi}_2$} ({7.5+2.2*cos(165))}, {7+2.2*sin(165)}) node {\tiny $\phi_1$} ({7.5+2.2*cos(195))}, {7+2.2*sin(195)}) node {\tiny $\overline{\phi}_4$} ({7.5+2.2*cos(225))}, {7+2.2*sin(225)}) node {\tiny $\phi_3$} ({7.5+2.2*cos(255))}, {7+2.2*sin(255)}) node {\tiny $\overline{\phi}_4$} ({7.5+2.2*cos(285))}, {7+2.2*sin(285)}) node {\tiny $\phi_2$} ({7.5+2.2*cos(315))}, {7+2.2*sin(315)}) node {\tiny $\overline{\phi}_1$} ({7.5+2.2*cos(345))}, {7+2.2*sin(345)}) node {\tiny $\phi_2$};

            \draw (5+8.2,3.2+4.5) node {\tiny $\phi_4$} (5+8.2,1.8+4.5) node {\tiny $\overline{\phi}_2$} (5+6.8,1.8+4.5) node {\tiny $\phi_3$} (5+6.8,3.2+4.5) node {\tiny $\overline{\phi}_1$};
    \end{tikzpicture}
    \caption{The map $\tau_1$ is reflection about the red arc and $\tau_2$ is reflection about the green arc.}
    \label{fig: reflections}
\end{figure}


It is straightforward to see that this 12-gon can be obtained by taking eight copies of the quadrilateral shown in Figure~\ref{quad} (left) and gluing them along the sides labelled by $y$ and $z$. To determine the shortest length of the $\{4,12\}$ filling pair, it therefore suffices to find the minimum value of $x + l$, where $x$ and $l$ are the side lengths of quadrilateral in Figure \ref{quad} (left). Now, the rest of the proof follows from Theorem \ref{thm 3.3} and from the fact that $\cosh^{-1}\left({\frac{7}{2}}\right) = 2 \cosh^{-1} \left({\frac{3}{2}}\right)$.
\begin{figure}[htbp]
    \centering
    \begin{tikzpicture}
        \foreach \x in {1,2,...,12}
         {
         \draw [bend left, -<-=.5] ({7.5+2.5*cos(15+\x*30)}, {7+2.5*sin(15+\x*30)}) to ({7.5+2.5*cos(15+(\x+1)*30)}, {7+2.5*sin(15+(\x+1)*30)});
         }

         \draw ({7.5+2.6*cos(0)}, {7+2.6*sin(0)}) node {\tiny$\alpha_3$} ({7.5+2.7*cos(30)}, {7+2.7*sin(30)}) node {\tiny$\beta_2^{-1}$} ({7.5+2.6*cos(60)}, {7+2.6*sin(60)}) node {\tiny$\alpha_2^{-1}$} ({7.5+2.6*cos(90)}, {7+2.6*sin(90)}) node {\tiny$\beta_1$} ({7.5+2.6*cos(120)}, {7+2.6*sin(120)}) node {\tiny$\alpha_4$} ({7.5+2.7*cos(150)}, {7+2.7*sin(150)}) node {\tiny$\beta_4^{-1}$} ({7.5+2.6*cos(180)}, {7+2.6*sin(180)}) node {\tiny$\alpha_1^{-1}$} ({7.5+2.6*cos(210)}, {7+2.6*sin(210)}) node {\tiny$\beta_2$} ({7.5+2.6*cos(240)}, {7+2.6*sin(240)}) node {\tiny$\alpha_2$} ({7.5+2.7*cos(270)}, {7+2.7*sin(270)}) node {\tiny$\beta_3^{-1}$} ({7.5+2.6*cos(300)}, {7+2.6*sin(300)}) node {\tiny$\alpha_4^{-1}$} ({7.5+2.6*cos(330)}, {7+2.6*sin(330)}) node {\tiny$\beta_4$};

         \draw ({7.5+2.2*cos(15))}, {7+2.2*sin(15)}) node {\tiny $\overline{\theta}$} ({7.5+2.2*cos(45))}, {7+2.2*sin(45)}) node {\tiny $\theta$} ({7.5+2.2*cos(75))}, {7+2.2*sin(75)}) node {\tiny $\overline{\theta}$} ({7.5+2.2*cos(105))}, {7+2.2*sin(105)}) node {\tiny $\overline{\theta}$} ({7.5+2.2*cos(135))}, {7+2.2*sin(135)}) node {\tiny $\theta$} ({7.5+2.2*cos(165))}, {7+2.2*sin(165)}) node {\tiny $\overline{\theta}$} ({7.5+2.2*cos(195))}, {7+2.2*sin(195)}) node {\tiny $\overline{\theta}$} ({7.5+2.2*cos(225))}, {7+2.2*sin(225)}) node {\tiny $\theta$} ({7.5+2.2*cos(255))}, {7+2.2*sin(255)}) node {\tiny $\overline{\theta}$} ({7.5+2.2*cos(285))}, {7+2.2*sin(285)}) node {\tiny $\overline{\theta}$} ({7.5+2.2*cos(315))}, {7+2.2*sin(315)}) node {\tiny $\theta$} ({7.5+2.2*cos(345))}, {7+2.2*sin(345)}) node {\tiny $\overline{\theta}$};
    \end{tikzpicture}
    \caption{}
    \label{fig:1.3}
\end{figure}
\end{proof}

\subsection*{Minimization of the value $x+l$}

\begin{figure}[htbp]
    \centering
    \begin{tikzpicture}
        \draw (0,-.5)--(0,3);
        \draw (0,3) to [bend left=20] (-1.5,3.5) to [bend left=10] (-3,2.5) to [bend left=10] (0,-.5);

        \draw (-.2,0.2) node {$\frac{\pi}{4}$} (.2,1.4) node {$z$} (-.25,2.65) node {$\frac{\pi}{2}$} (-.75, 3.5) node {$\frac{l}{2}$} (-2.25, 3.1) node{$x$} (-2.25,2.4) node{$\theta/2$} (-1.6,1.1) node{$y$};

        \draw (1.6,1.25) to [bend left=10] (3.4, -.75) to [bend left=10] (5.4, 1.25) to [bend left=10] (3.6, 3.25) to [bend left=10] (1.6,1.25); 

        \draw (2,1.25) node {$\theta$} (4.9,1.25) node {$\theta$} (3.6,2.8) node {$\theta$} (3.4, -.3) node {$\theta$} (2.4,.25) node {$l$} (4.5,.25) node {$l$} (4.6,2.25) node {$l$} (2.5,2.25) node {$l$};
    \end{tikzpicture}
    \caption{}
    \label{quad}
\end{figure}

By assumption, one side of the quadrilateral has length \(\tfrac{l}{2}\), and the interior angles are \(\tfrac{\pi}{4}, \tfrac{\pi}{2}, \pi-\theta\), and \(\tfrac{\theta}{2}\) (see Figure~\ref{quad}). Let the lengths of the other three sides be \(x,y,z\). Our objective is to minimize \(x + l\) while satisfying the geometric constraints of the two quadrilaterals in Figure~\ref{quad}.

For a hyperbolic triangle with angles \(A,B,C\) and and opposite side lengths \(a,b,c\), the following identity holds.
\begin{equation}
\label{hypertriangle}
    \cosh a=\frac{\cos A +\cos B \cos C }{\sin B \sin C }.
\end{equation}

Now using the identity \eqref{hypertriangle} in the quadrilateral (right side) in Figure~\ref{quad}, the side length \(l\) is given by
\begin{equation}
    \label{triangle1}
    \cosh{l} = \frac{\cos\left(\frac{\theta}{2}\right)+\cos\theta \cos\left(\frac{\theta}{2}\right)}{\sin\theta\sin\left(\frac{\theta}{2}\right)}
\end{equation}

\begin{figure}[htbp]
    \centering
    \begin{tikzpicture}
        \draw (0,-1)--(0,3);
        \draw (0,3) to [bend left=20] (-2,3.8) to [bend left=10] (-4,2.5) to [bend left=10] (0,-1);

        \draw (0,3) to [bend left=25] (-4,2.5);

        \draw (-.2,-.3) node {$\frac{\pi}{4}$} (.2,1) node {$z$} (-.6,2.9) node {\tiny$\theta_2$} (-.6,2.3) node {\tiny $\pi/2-\theta_2$}  (-.9, 3.5) node {$\frac{l}{2}$} (-2.95, 3.3) node{$x$} (-3.35,2.55) node {\tiny$\theta_1$} (-2.5,2) node {\tiny$\theta/2-\theta_1$} (-2,.75) node{$y$} (-2,3.3) node {\tiny$\pi-\theta$} (-2,2.45) node {$w$};
    \end{tikzpicture}
    \caption{}
    \label{quad2}
\end{figure}

For the left quadrilateral in Figure \ref{quad}, we introduce an auxiliary geodesic of length $w$, see figure \ref{quad2}. By applying Equation~\eqref{hypertriangle} to the two triangles formed in Figure~\ref{quad2}, it follows that
\begin{equation}
\label{triangle2}
    \cosh w =\frac{-\cos \theta +\cos \theta_1 \cos \theta_2 }{\sin \theta_1 \sin \theta_2 }=\frac{\frac{\sqrt{2}}{2}+\cos(\frac{\theta}{2}-\theta_1)\sin \theta_2 }{\sin(\frac{\theta}{2}-\theta_1)\cos \theta_2 }.
\end{equation}

Similarly, $l$ and $x$ can also be expressed in terms of the three angles of the triangle:
\begin{equation}
\label{triangle3}
\begin{aligned}
    \cosh x &=\frac{\cos \theta_2 -\cos \theta \cos \theta_1 }{\sin \theta \sin \theta_1}\\
    \cosh\left(\frac{l}{2} \right)&=\frac{\cos \theta_1 -\cos \theta \cos \theta_2}{\sin \theta \sin \theta_2 }.
\end{aligned}
\end{equation}

However, theoretically deriving the minimum of $x+l$ remains challenging. The uniqueness of the minimum for $x+l$ follows from the convexity of the length function $\text{length}_X(\alpha,\beta)$, which makes it possible to numerically compute the minimum point. We can formulate the numerical computation method as follows: with five variables $l,\theta,x,y,z$ and four constraints (one constraint in \eqref{triangle1}, one constraint in \eqref{triangle2} and two constraints in \eqref{triangle3}), we can use $\theta$ to solve the remaining variables. This reduces $x+l$ to a $\theta$-dependent function, whose minimum can be numerically found via gradient descent. This numerical method can be referenced in Section 3.2 of \cite{AnIhringerIrmer2025}. We implement this numerical computation using Python \footnote{Code associated with this work is available on the arXiv version of this paper}. Subsequently, we can use mathematical tools such as SageMath to determine the exact values of the minimum point.

With the exact value of the minimum point. To rigorously prove the result, we employ the method of Lagrange multipliers for constrained optimization. For a differentiable objective function f: $\mathbb{R}^n \rightarrow \mathbb{R}$ with equality constraints $g_i(\mathbf{x})=0$, we construct the Lagrangian:

\[
\mathcal{L}(\mathbf{x},\boldsymbol{\lambda})=f(\mathbf{x})-\sum_{i=1}^{m}\lambda_{i}g_{i}(\mathbf{x}),
\]
where stationary points satisfy $\nabla_{\!\mathbf{x}}\mathcal{L}=\mathbf{0}$ and $g_i(\mathbf{x})=0$. For more details about Lagrange multipliers, see Section 14.8 of \cite{hass2017thomas}. Applying this framework to our problem with $f=x+l$ and the four constraints  establishes the necessary and sufficient optimality conditions, we have the following theorem:

\begin{theorem} \label{thm 3.3}
        The minimum point of $f(x,l,\theta,\theta_1,\theta_2)=x+l$ under the four constraints: 
        \begin{equation}
            \begin{aligned}
            g_1&=\frac{-\cos \theta+\cos \theta_1 \cos \theta_2}{\sin \theta_1 \sin \theta_2}-\frac{\frac{\sqrt{2}}{2}+\cos(\frac{\theta}{2}-\theta_1)\sin \theta_2}{\sin(\frac{\theta}{2}-\theta_1)\cos \theta_2}=0\\
                g_2&=\cosh{l} - \frac{\cos\left(\frac{\theta}{2}\right)+\cos\theta \cos\left(\frac{\theta}{2}\right)}{\sin\theta\sin\left(\frac{\theta}{2}\right)}=0\\
                g_3&=\cosh x-\frac{\cos \theta_2-\cos \theta\cos \theta_1}{\sin \theta\sin \theta_1}=0\\
    g_4&=\cosh(\frac{l}{2})-\frac{\cos \theta_1-\cos \theta\cos \theta_2}{\sin \theta \sin \theta_2}=0
            \end{aligned}
        \end{equation}
        is given by
        $$ p_0=\left(\cosh^{-1}\left(\frac{7}{2}\right),\cosh^{-1}\left(\frac{3}{2}\right),\cos^{-1}\left(\frac{1}{5}\right),\tan^{-1}\left(\frac{2\sqrt{6}}{41}\right),\tan^{-1}\left(\frac{6\sqrt{30}}{25}\right)\right).$$
\end{theorem}
\begin{proof}
    Let 
    \begin{equation*}
        \mathcal{L}(x,l,\theta,\theta_1,\theta_2,\boldsymbol{\lambda})=x+l-\sum_{i=1}^4\lambda_ig_i,
    \end{equation*}
     where $\boldsymbol{\lambda}=(\lambda_1,\lambda_2,\lambda_3,\lambda_4)$.
It can be obtained by the Lagrange multipliers that $\boldsymbol{p_0}$ is the minimum point if and only if there exists $\boldsymbol{\lambda_0}=(\lambda_1^0,\ldots,\lambda_n^0)$ such that at $(\boldsymbol{p_0},\boldsymbol{\lambda_0})$, the following equalities hold:
\begin{equation*}
    \begin{aligned}
    &\mathcal{L}(\boldsymbol{p_0},\boldsymbol{\lambda_0})=0,~\frac{\partial \mathcal{L}}{\partial \lambda_i}=0\\
        \frac{\partial \mathcal{L}}{\partial x}&=\frac{\partial \mathcal{L}}{\partial l}=\frac{\partial \mathcal{L}}{\partial \theta}=\frac{\partial \mathcal{L}}{\partial \theta_1}=\frac{\partial \mathcal{L}}{\partial \theta_2}=0
    \end{aligned}
\end{equation*}
$\boldsymbol{\lambda_0}$ can be found through calculation using Mathematica \footnote{Code associated with this work is available on the arXiv version of this paper}, where $\boldsymbol{\lambda_0} =(\frac{70}{341},\frac{3}{5 \sqrt{5}},\frac{2}{3 \sqrt{5}},\frac{14}{5})$.
\end{proof}


\section{Length of arbitrary filling pairs in genus two}
In this section, we show that the length of any filling pair on a hyperbolic surface of genus two is bounded from below by $L_0$, where $L_0$ is the shortest length of a minimal filling pair of type $\{ 4,12 \}$. To achieve this, we need some results on graph theory. First, we recall some basic definitions.

\begin{definition}
    A subgraph $H\subset G$ of a graph $G$ is called \textit{spanning} if it contains all the vertices of $G$. A subgraph $H\subset G$ is called \textit{spread} if any two edges $e$ and $e'$ emanating from a vertex are not consecutive in $G$.
\end{definition}

We use the following result due to Jonah Gaster (see Section 2 \cite{Gaster}).
\begin{lemma}\label{lem 4.2}
    Let $(\alpha,\beta)$ be a filling pair of simple closed curves in minimal position on a closed orientable surface $S$. If $\alpha$ is non-separating, then the dual graph to $\alpha \cup \beta$ admits a spread spanning tree. If $\alpha$ is separating, the dual graph admits a spread spanning forest with two components.
\end{lemma}

\subsection*{An observation on the number of sides of the complementary polygons of a filling pair on \texorpdfstring{$S_2$}{TEXT}.} Let $S_2\setminus(\alpha\cup \beta) = \bigsqcup_{i=1}^f D_i$. Let the number of sides in $D_i$ be $|D_i|=4k_i + j_i$, where $k_i\geq 1$ and $j_i=0 \text{ or } 2 , i=1,\dots,f$. Considering $(\alpha\cup \beta)$ as a 4-valent graph, we have
\begin{align}\label{eq 2.1}
    i(\alpha,\beta)&=\frac{1}{4}\sum_{i=1}^f |D_i| \nonumber \\
    &=\sum_{i=1}^f k_i + \frac{1}{4}\sum_{i=1}^f j_i
\end{align}
Also, by using the Euler's characteristic arguments,
\begin{equation}\label{eq 2.2}
    i(\alpha,\beta)=f+2.
\end{equation}
Based on equations \eqref{eq 2.1} and \eqref{eq 2.2}, one of the following cases occur:
\begin{itemize}
    \item $k_i=1$ for all $i\in \{ 1,\dots,f\}$ and $j_i=2$ for exactly four $i\in \{ 1,\dots,f\}$.
    \item $k_i=2$ for one $i\in \{ 1,\dots,f\}$, others $k_i$ are 1 and $j_i=2$ for exactly two $i\in \{ 1,\dots,f\}$.
    \item $k_i=2$ for two $i\in \{ 1,\dots,f\}$.Then the others $k_i$ are 1 and all $j_i$'s are zero.
    \item $k_i=3$ for one $i$, $k_i=1$ for other $i$'s and all $j_i$'s are zero.
\end{itemize}
From the above observation, we have the following lemma.
\begin{lemma}
        Any filling pair on $S_2$ has the property that at least one of the polygons in its complement contains $4k$ sides, for some $k \in \{1,2,3\}$, with one possible exception where the complement contains $4$ polygons, each with $6$ sides.
\end{lemma}

Now we prove the main theorem of this section.
\begin{theorem}
    Let $(\alpha,\beta)$ be any filling pair on a hyperbolic surface $X$ of genus two. Then 
    $$l_X(\alpha,\beta)\geq L_0,$$
    where $L_0= 6 \cdot \cosh^{-1}(7/2)$, the shortest length of the minimal filling pair of type $\{ 4,12 \}$.
\end{theorem}
\begin{proof}
    First, we consider the possible exceptional case, that is $X\setminus (\alpha \cup \beta)$ consists of four hyperbolic hexagons. In this case, we take the hexagons in pairs and attach them along a common side. This will produce two polygons $D_1$ and $D_2$ with total area $4\pi$. The angles at the common vertices are all $\pi$ and hence we may consider each of the attached polygons as an octagon. Also, one can see that $l_X(\alpha,\beta)\geq \mathrm{Perim}(D_1)+ \mathrm{Perim}(D_2)$. By the proof of Theorem \ref{thm: length}, we have $l_X(\alpha,\beta)\geq 8\cdot\cosh^{-1}\left(\sqrt{2}+1\right)$.

    Now, suppose that at least one of the complementary polygons has $4k$ sides, for some $k \in \{1,2,3\}$. To prove the theorem in this case, we apply Lemma~\ref{lem 4.2}.

    If both of the curves are separating, then by Lemma~\ref{lem 4.2}, the dual graph of $(\alpha,\beta)$ admits a spread spanning forest with two components. By attaching the polygons along these trees, we obtain two octagons (see Section~3 of \cite{Gaster} for a detailed discussion).

    If one of the curves non-separating, we fix the $4k$-gon and attach the remaining polygons according to the spanning tree obtained from Lemma~\ref{lem 4.2}. In this way, we obtain a $4k$-gon (the fixed one) and a $(4-k)$-gon. Applying Theorem~\ref{thm: length} or Theorem~\ref{thm 3.3}, we then have
    $$
    l_X(\alpha,\beta) \geq 8\cdot\cosh^{-1} \! \left(\sqrt{2}+1\right)
    \quad \text{or} \quad 
    l_X(\alpha,\beta) \geq L_0.
    $$
    This completes the proof.
\end{proof}

\bibliographystyle{alpha}
\bibliography{bibliography}

\end{document}